\documentclass[10 pt, reqno]{amsart}
\usepackage{amsmath,amssymb,amsthm,amscd,amsfonts,mathrsfs}
\usepackage{latexsym}
\usepackage{tikz}
\usepackage{stmaryrd}
\usepackage{pgf,tikz}
\usepackage{graphicx}
\usepackage{color}
\usepackage{enumitem}
\usepackage{mathpazo}
\usepackage[T1]{fontenc}
\usepackage{microtype}
\usepackage{soul}
\usepackage{mathtools}

\usepackage[colorinlistoftodos]{todonotes}
 \presetkeys{todonotes}%
{inline,backgroundcolor=gray!20,bordercolor=gray!30}{}
\tikzset{/tikz/notestyleraw/.append style={text=black}}

\usepackage{fullpage}
\newtheorem{thm}{Theorem}[section]

\newtheorem{lem}[thm]{Lemma}
\newtheorem{defn}[thm]{Definition}
\newtheorem{prop}[thm]{Proposition}

\newtheorem{cor}[thm]{Corollary}

\newtheorem{rmk}[]{Remark}
\newcommand{\be}{\begin{eqnarray}}
\newcommand{\ee}{\end{eqnarray}}
\newcommand{\ben}{\begin{eqnarray*}}
\newcommand{\een}{\end{eqnarray*}}
\newcommand{\beal}{\begin{aligned}}
\newcommand{\enal}{\end{aligned}}
\newcommand{\beq}{\begin{equation}}
\newcommand{\eeq}{\end{equation}}

\newcommand{\R}{\mathbb{R}}

\newcommand{\N}{\mathbb{N}}

\newcommand{\Om}{\Omega}

\usepackage[unicode=true]{hyperref}
\hypersetup{
     colorlinks,
     linkcolor={black!10!blue},
     linkbordercolor = {black!100!blue},
     citecolor={red}
}

\usepackage{enumitem}
\makeatletter
\def\namedlabel#1#2{\begingroup
    #2%
    \def\@currentlabel{#2}%
    \phantomsection\label{#1}\endgroup
}
\makeatother

\title{Vanishing Discount Limits for First-Order Fully Nonlinear Hamilton-Jacobi Equations on Noncompact Domains}
\author{Son N.T. Tu$^\dagger$ and Jianlu Zhang$^*$}

\address{$^\dagger$Department of Mathematics, Baylor University\\
Waco, Texas 76706, USA}
\email{son\_tu@baylor.edu}

\address{$^*$State Key Laboratory of Mathematical Sciences, Academy of Mathematics and Systems Science, Chinese Academy of Sciences, Beijing 100190, China}
\email{jellychung1987@gmail.com}

\thanks{{\it Statements and Declarations: }The authors declare no competing interests. The work of Son Tu was supported in part by NSF DMS Grant 220472. The work of J. Zhang is supported by the National Key R\&D Program of China (No. 2022YFA1007500) and the National
Natural Science Foundation of China (No. 12231010).}
\subjclass[2020]{
35D40, 
70H20, 
35J60, 
37J40, 
49L25, 
37K99, 
}
\keywords{viscosity solution, Hamilton-Jacobi equations, effective Hamiltonian, Mather measure}
\date{\today}

\numberwithin{equation}{section}

\newif\ifsol 
\soltrue 
\solfalse

\begin{document}

\begin{abstract} 
We study the asymptotic behavior  of solutions to the fully nonlinear Hamilton-Jacobi equation \( H(x, Du, \lambda u) = 0 \) in \( \mathbb{R}^n \) as \( \lambda \to 0^+ \). Under the assumption that the Aubry set is localized, we employ a variational approach to derive limiting Mather-type measures and formulate a selection principle. Central to our analysis is a modified variational formula that bridges global and local state-constraint solutions, thereby extending localization techniques to the nonlinear framework.
\end{abstract}

\maketitle

\section{Introduction}
\subsection{Motivation}
Let $H(x,p,u):\R^n\times \R^n\times \R \to \R$ be a continuous Hamiltonian. We study the asymptotic behavior, as $\lambda \to 0^+$, of the viscosity solutions to the contact Hamilton-Jacobi equations
\begin{equation}\label{eq:DP}
	H(x,Du, \lambda u) = c(H) \quad\text{in}\;\R^n
\end{equation}
posed in $\R^n$. Here $c(H)$ is the critical value or ergodic constant of $H$ in $\R^n$, defined as 
\begin{equation}\label{eq:c(H)}
	c(H) = \inf \{a\in \R: H(x,Du, 0) = a\;\text{admits a continuous viscosity subsolution in}\; \R^n\}. 
\end{equation}
As $\lambda\to 0^+$, the expected limiting problem of \eqref{eq:DP} should be
\begin{equation}\label{eq:E}
    H(x,Du, 0) = c(H) \quad\text{in}\;\R^n. 
\end{equation}

Although \eqref{eq:DP} may admit multiple solutions, a maximal solution $u_\lambda$ can be uniquely defined. We aim to show \( u_\lambda \to u \) as \( \lambda \to 0^+ \), where \( u \) solves \eqref{eq:E}. Since \eqref{eq:E} may have multiple solutions, this convergence serves as a \emph{selection principle}. Such a \emph{vanishing discount} problem for the \emph{discounted Hamiltonians} \( H(x, p, u) = H_0(x, p) + u \) has been well studied on bounded domains \cite{davini_convergence_2016, ishii_vanishing_2017-1, ishii_vanishing_2017, mitake_selection_2017}, and the unbounded case is treated in \cite{capuzzo_dolcetta_vanishing_2023,ishii_vanishing_2020}.

\medskip

Motivated by \cite{ishii_vanishing_2020}, we study \emph{contact Hamiltonians} (e.g., fully nonlinear \cite{lions_generalized_1982}) on noncompact spaces, extending the discounted case. In contrast with the duality method used in \cite{ishii_vanishing_2020}, we obtained the selection principle from weak convergence of Mather measures built via minimizing curves of certain variational formula. The curve-based measures contact Hamiltonians was developed in \cite{wang_convergence_2021}, extending the discounted case \cite{davini_convergence_2016}. 
State-constraint solutions are used as a tool (Section \ref{sec:state-constraint}). However, our representation formula differs significantly from the discounted case, complicating localization, i.e., to show that the global solution matches the state-
We introduce a modified representation that maintains this localization.

\subsection{Setting}
We adopt assumptions from \cite{ishii_vanishing_2020}, suitably modified to fit our fully nonlinear framework.

\begin{description}[style=multiline, labelwidth=1cm, leftmargin=2cm]   
    \item[\namedlabel{itm:assumptions-h1}{$(\mathcal{H}_1)$}] $H(x,\xi, u)\in C(\R^n\times \R^n \times \R ;\R)$ is nondecreasing in $u$, \emph{convex} in $\xi$ for each $(x,u) \in \R^n\times \R$, and is coercive in $|\xi|$ locally in $(x,u)$: For any compact sets \( K \subset \R^n \), \( I \subset \R \),  
    \begin{equation}\label{eq:assumptions-LocalCoercive}
        \lim_{r \to \infty} \inf \{ H(x, p, u) : x \in K, |p| \ge r, u \in I \} = +\infty.
    \end{equation}
    \item[\namedlabel{itm:assumptions-h2}{$(\mathcal{H}_2)$}] There exists $\varepsilon>0$ such that 
    \begin{equation}\label{eq:assumptions-Aubry}
        \limsup_{|x|\to \infty} \max_{|p|\leq \varepsilon} H(x,p,0) < \max_{x\in \R^n} \min_{p\in \R^n} H(x,p,0) < \infty. 
    \end{equation}
\end{description}

Solutions to \eqref{eq:E} are generally non-unique; in the bounded case, they can be characterized via the Aubry set (see Section \ref{sec:cHAubry}). Assumption \ref{itm:assumptions-h2}, from \cite{ishii_vanishing_2020}, ensures \( c(H) \) is well-defined and the Aubry set is compact. Under \ref{itm:assumptions-h1}--\ref{itm:assumptions-h2}, \eqref{eq:DP} may still lack uniqueness on unbounded domains, but a maximal locally Lipschitz solution exists. We also assume \( H \) has superlinear growth in \( p \), without loss of generality (Proposition \ref{prop:SuperlinearReduction}).
\medskip

To analyze the dynamics, we impose stronger monotonicity in \(u\) (local in \(p\)), with \(R > 0\) below.
\begin{description}[style=multiline, labelwidth=1cm, leftmargin=2cm]   
    \item[\namedlabel{itm:assumptions-h3}{$(\mathcal{H}_3)$}] 
    Assume \( H(x, p, u) \) is differentiable in \( u \), 
    with \( \underline{\kappa}_R \leq \partial_u H(x, p, u) \leq \overline{\kappa}_R\) for \( (x, p) \in \mathbb{R}^n \times \overline{B}_R(0)\),  and that there exists a modulus \( \omega_R \) such that
    \begin{align}\label{eq:assumptions-DiffC1}
		\left|\frac{H(x,p,u) - H(x,p,0)}{u} - \partial_u H(x,p,0)\right| \leq \omega_R(|u|), \quad (x,p) \in \overline{B}_R(0) \times \overline{B}_R(0). 
	\end{align}
\end{description}
Assumption \ref{itm:assumptions-h3} is adopted for convenience in stating the main result. Many intermediate steps (such as the representation formula) only require the lower bound on \( \partial_u H \), not the upper. The modulus is used solely to justify convergence via weak convergence of measures. 

A stronger condition is a uniform bound on the derivative, similar to the discounted case, though not always required to derive the selection principle.
\begin{description}[style=multiline, labelwidth=1cm, leftmargin=2cm]   
    \item[\namedlabel{itm:assumptions-h4}{$(\mathcal{H}_4)$}] 
    Assume \( H(x, p, u) \) is differentiable in \( u \), 
    with 
    \( \partial_u H(x, p, u) \leq \overline{\kappa}\) for \( (x, p) \in \mathbb{R}^n \times \R^n\). 
\end{description}
We apply the Lax-Oleinik semigroup with initial data \( u_\lambda \), producing a value function from a finite-horizon optimal control problem. Since \( u_\lambda(x) + c(H)t \) is also a solution, if a comparison principle holds, it yields the desired variational formula for $u_\lambda(x)$. This enables constructing associated measures via minimizing curves. We present three different assumption sets, each ensuring a comparison principle (see Subsection \ref{subsection:discussion}).

\begin{description}[style=multiline, labelwidth=1cm, leftmargin=2cm]   
    \item[\namedlabel{itm:assumptions-pp1}{$(\mathcal{P}_1)$}] Assume \ref{itm:assumptions-h4} and for some \(\theta \in (0,1)\), there exists \(C_\theta > 0\) such that
	\begin{equation}\label{eq:assumptions-MonotoneInP}
		H(x,\theta p,u) \leq H(x,p,u) + C_\theta, 
        \qquad
        (x,p,u)\in \R^n\times \R^n\times \R.
	\end{equation}

    \item[\namedlabel{itm:assumptions-pp2}{$(\mathcal{P}_2)$}] Assume \ref{itm:assumptions-h4}, \(H(x, p, u)\) is convex in \((p, u)\), and for some \(\theta \in (0,1)\), there is \(C_\theta > 0\) such that:
	\begin{equation}\label{eq:assumptions-JointlyConvex} 
		 H(x,p,\theta u) \leq C \quad\Longrightarrow\quad H(x,p, u) \leq C_\theta, \qquad (x,p,u) \in \R^n\times\R^n\times \R.
	\end{equation}

    \item[\namedlabel{itm:assumptions-pp3}{$(\mathcal{P}_3)$}] 

    \(H\) satisfies \ref{itm:assumptions-h3}, a uniform upper bound, bounded below, and uniform coercivity:
    \begin{align}
        &\,\sup \{H(x,p,u): x\in \R^n, |p|\leq R, |u|\leq C\} < \infty 
        \qquad\quad\quad\,
        \text{for any}\;R>0, C>0,
        \label{eq:assumptions-UniformBoundWithP} \\
        &\,\inf \{H(x,p,u_0): (x,p) \in\R^n\times \R^n\} > -\infty 
        \qquad\qquad\qquad\;
        \text{for any fixed}\;u_0\in \R, \label{eq:eq:assumptions-UniformBoundBelow} \\
         &\lim_{r \to \infty} \inf \{ H(x, p, u) : x \in \R^n, |p| \geq r, |u|\leq C \} = +\infty
        \qquad
        \text{for any}\;C>0. \label{eq:assumptions-UniformCoercive} 
    \end{align}
\end{description}
\subsection{Main Results} We present the main results, including the selection principle and a localization result inspired by \cite[Proposition 5.2]{ishii_vanishing_2020}, highlighting key differences between the nonlinear and discounted cases.

\begin{thm}\label{thm:selection-principle} Assume \ref{itm:assumptions-h1}--\ref{itm:assumptions-h3} and one of \ref{itm:assumptions-pp1}--\ref{itm:assumptions-pp3}. Then the maximal solution \( u_\lambda \in C(\mathbb{R}^n) \) to \eqref{eq:DP} converges locally uniformly to a solution \( u_0 \in C(\mathbb{R}^n) \) of \eqref{eq:E}, where $u_0 = \sup \mathcal{E}$, with
\[
    \mathcal{E} := \left\{ w \in C(\mathbb{R}^n) : w \text{ is a subsolution of \eqref{eq:E} and } \int_{\R^n\times\R^n} w(x)\, \partial_u L(x,v,0)\, d\mu \geq 0 \ \text{for all}\  \mu \in \mathfrak{M} \right\},
\]
and \( \mathfrak{M} \) is a set of probability measures analogous to classical Mather measures (see Definition \ref{defn:SMather}).
\end{thm}

One key tool is the state-constraint solution \( \vartheta_{\lambda,R} \) to \eqref{eq:DP} on \( B_R(0) \): it is a subsolution in \( B_R(0) \) and a supersolution on \( \overline{B}_R(0) \) (see Section \ref{sec:state-constraint}). 
In the discounted case \cite{ishii_vanishing_2020}, these solutions admit an optimal control formula with \( L = H^* \), the Legendre transform, minimizing over absolutely continuous curves \( \gamma(0) = x \):
\begin{equation*}
	u_\lambda(x) = \inf \int_{-\infty}^0 e^{\lambda s} \big( L(\gamma, \dot{\gamma}, 0) + c(H) \big)\, ds, 
    \quad 
	\vartheta_{\lambda,R}(x) = \inf_{|\gamma(s)|\leq R}
	 \int_{-\infty}^0 e^{\lambda s} \big( L(\gamma, \dot{\gamma}, 0) + c(H) \big)\, ds,
\end{equation*}
and thus it is clear that \( u_\lambda(x) = \vartheta_{\lambda,R}(x) \) if minimizers for $u_\lambda(x)$ remains in \( \overline{B}_R(0) \). For the nonlinear case, these functions are represented similarly, but with a modified exponential weight:
\begin{equation}\label{eq:representationsIntro}
    u_\lambda(x) = \int_{-\infty}^0 e^{\lambda \beta(s)} \left( L(\gamma, \dot{\gamma}, 0) + c(H) \right) \, ds, 
    \quad
    \vartheta_{\lambda,R}(x) = \int_{-\infty}^0 e^{\lambda \alpha(s)} \left( L(\gamma, \dot{\gamma}, 0) + c(H) \right) \, ds,    
\end{equation} 
where the discount indices \( \alpha \) and \( \beta \) depend on \( u_\lambda \) and \( \vartheta_\lambda \) (see Remark \ref{remark:index}). 
This complicates the deduction of \( u_\lambda(x) = \vartheta_{\lambda,R}(x) \), even if minimizers for \( u_\lambda(x) \) remain in \( \overline{B}_R(0) \). However, introducing a new discount index recovers this equality.

\begin{thm}\label{thm:localization} Assume \ref{itm:assumptions-h1}--\ref{itm:assumptions-h3} and one of \ref{itm:assumptions-pp1}--\ref{itm:assumptions-pp3}. 
Let $\vartheta_{\lambda,R}$ be the state constraint solution of \eqref{eq:DP} in $B_R(0)$. 
For any $z\in\R^n$, there exists  $R_z > 0$ and  $\lambda_z \in (0,1)$ such that $u_\lambda(z) = \vartheta_{\lambda,R}(z)$ for $R\geq R_z$ and $\lambda \in (0,\lambda_z)$. 
\end{thm}
This localization and quantitative approximation idea is related to \cite{kim_state-constraint_2020}, which provides optimal error estimates in the discounted case between global and state-constraint solutions. Theorem \ref{thm:localization} offers a short proof of Theorem \ref{thm:selection-principle} by using the established results on vanishing discount for state-constraint problems on a bounded domain, as presented in \cite[Theorem 1.1]{tu_generalized_2024}. 

\begin{rmk}\label{remark:index} Regarding the representations \eqref{eq:representationsIntro}, one can also define \( \alpha = \beta = \partial_u L(\gamma, \dot{\gamma}, 0) \), as in \cite{chen_convergence_2024, tu_generalized_2024, wang_convergence_2021, wang_negative_2024}. However, this leads to additional error terms in the presentation formulas (differing for \( u_\lambda \) and \( \vartheta_{\lambda, R} \)) that vanish as \( \lambda \to 0^+ \), which suffices for studying the vanishing discount problem on bounded domains, but not for our localization result. To overcome this, we introduce slightly modified indices in Definition \ref{defn:NewIndices}.
\end{rmk}

\subsection{Discussion on the Assumptions} \label{subsection:discussion} 
In this subsection, we justify the additional assumptions \ref{itm:assumptions-pp1}--\ref{itm:assumptions-pp3} and clarify when \ref{itm:assumptions-h3} and \ref{itm:assumptions-h4} differ in a nontrivial way. 

\subsubsection{Additional Structural Assumptions}
We give examples for \ref{itm:assumptions-pp1} or \ref{itm:assumptions-pp2} with  
\[
H(x,p,u) = H_0(x,p) + \phi(x,u), \quad (x,p,u) \in \R^n \times \R^n \times \R,
\]
where \( H_0(x,p) \) is continuous, convex, and locally uniformly coercive (i.e., satisfying \ref{itm:assumptions-h1}--\ref{itm:assumptions-h3} with $u=0$). 
\begin{itemize}
    \item  Let \( H_0(x,p) = \mathfrak{h}(p) - f(x) \), where \( \mathfrak{h} \) is convex, coercive, and homogeneous of degree \( \tau > 0 \), and \( f \in C(\R^n) \) 
    does not attain its minimum at infinity. Assume \( \phi \in C^1(\R^n \times \R) \) with \( \underline{\kappa} \leq \phi_u(x,u) \leq \overline{\kappa} \). Then \eqref{eq:assumptions-MonotoneInP} in \ref{itm:assumptions-pp1} is satisfied, since for \( \theta \in (0,1) \),
    \begin{equation*}
        H(x,\theta p,u) \leq H(x,p,u) + (1 - \theta^\tau) \mathfrak{h}_0,  
    \end{equation*}
    where $\mathfrak{h}_0 = \min_{p \in \R^n} \mathfrak{h}(p)$.
    This homogeneity assumption is also adopted in \cite{mitake_rate_2019}.

    \item If \( \phi(x,u) = \phi(x)\cdot u \) with \( \phi \in C^1(\R^n) \) and \( 0 < \phi(x) \leq \overline{\kappa} \), and if \( H_0(x,p) \geq -c_0 \) is bounded from below, then \eqref{eq:assumptions-JointlyConvex} in \ref{itm:assumptions-pp2} holds. Indeed, for \( \theta \in (0,1) \), if \( H(x,p,\theta u) = H_0(x,p) + \phi(x)\cdot \theta u \leq C \), then
    \begin{align*}
	   H_0(x,p) + \alpha(x)\cdot u 
        \leq 
        \theta^{-1}\big( H_0(x,p) + \alpha(x)\cdot u +  c_0 \big) \leq \theta^{-1}(C+c_0).
    \end{align*}
    The convexity of \( (p, u) \mapsto H(x, p, u) \) has previously appeared in \cite[Assumption (A5)]{jing_generalized_2020}. 
\end{itemize}
Assumptions \ref{itm:assumptions-pp1} and \ref{itm:assumptions-pp2} allow constructing a subsolution \( u \) to \( H(x, Du, \lambda u_\lambda) = C \) in $\R^n$ for some constant, with \( u_\lambda(x) - u(x) \to \infty \) as \( |x| \to \infty \), which suffices to apply a comparison principle for unbounded solutions (Theorem \ref{thm:ComparisonIshiiMod}, adapted from \cite[Theorem 4.1]{ishii_asymptotic_2008}) and derive the variational formula for \( u_\lambda \).

\medskip 

Assumption \ref{itm:assumptions-pp3} resembles those in \cite{fathi_weak_2007}, which studied a weak KAM theory on noncompact domains for classical Hamiltonians.
Essentially, \eqref{eq:assumptions-UniformBoundWithP} and \eqref{eq:eq:assumptions-UniformBoundBelow} imply \( u_\lambda \) is uniformly bounded and Lipschitz. While Perron's method directly constructs such a solution in \( \mathrm{BUC}(\R^n) \), the maximal continuous solution need not coincide with it a priori. Under these assumptions, the standard comparison principle in \( \mathrm{BUC}(\R^n) \) yields the desired variational formula for \( u_\lambda \), and \eqref{eq:assumptions-UniformBoundWithP} ensures uniformly bounded velocities for minimizing curves.

\subsubsection{Assumptions on Monotonicity} The uniform monotonicity condition in \ref{itm:assumptions-h4} naturally appears in settings analogous to the discounted case, particularly when the Hamiltonian takes the form 
\( H(x,p,u) = H_0(x,p) + \phi(x) \cdot u \), with \( \phi(x) \in (\underline{\kappa}, \overline{\kappa}) \) as previously discussed. 
\medskip

Assumption \ref{itm:assumptions-h3} captures a subtler structure where monotonicity is non-uniform. A representative case is 
\begin{equation*}
    H(x,p,u) = H_0(x,p) + \big(|p|^2+1\big) \cdot \left(\arctan(u)+\pi\right) + u. 
\end{equation*}
This Hamiltonian satisfies assumption \ref{itm:assumptions-h3} but fails to satisfy \ref{itm:assumptions-h4}. When $H_0$ additionally meets assumption \ref{itm:assumptions-pp3}, this example satisfies the conclusion of Theorem \ref{thm:selection-principle}. Consequently, a selection principle applies in this case, and it differs fundamentally from the discounted scenario.
This structure reflects a nontrivial coupling between the gradient and the solution. It also highlights the role of minimizing curves with bounded velocity, a key to the variational formulation of the value function, within our framework.

\subsection{Literature Review}

The discounted case was originally introduced as a vanishing discount approximation to \eqref{eq:E}, posed on the torus and interpreted as a cell problem in homogenization. Subsequence convergence was shown in \cite{lions_homogenization_1986}, with full-sequence convergence later established in \cite{davini_convergence_2016} via Mather measures and in \cite{mitake_selection_2017} via the nonlinear adjoint method, including degenerate second-order cases (which is generalized to the contact case in \cite{zhang_limit_2024}). A duality-based approach to constructing Mather measures was developed in \cite{ishii_vanishing_2017-1,ishii_vanishing_2017}, extending to state-constraint \cite{tran_hamilton-jacobi_2021}, Dirichlet, and Neumann boundary conditions (see also \cite{al-aidarous_convergence_2016} for earlier work on the Neumann condition, and \cite{gomes_duality_2005} for the introduction of the duality framework). The whole-space case $\R^n$ was initiated in \cite{ishii_vanishing_2020}, which inspired the present work. A recent extension to nonlocal equations appears in \cite{davini_vanishing_2025}. Vanishing discount problems on variable domains were further studied in \cite{tu_vanishing_2022} for first-order cases and \cite{bozorgnia_regularity_2024} for second-order equations with super-quadratic Hamiltonians. Convergence rates remain largely unexplored, except for certain special cases treated in \cite{mitake_weak_2018}. \medskip

The vanishing discount limit for contact Hamilton--Jacobi equations was first explored in \cite{chen_vanishing_2019}, followed by several developments in \cite{ 
chen_convergence_2023, 
chen_convergence_2024, 
davini_vanishing_2021,
ni_multiple_2023, 
wang_convergence_2021,
wang_negative_2024, 
zavidovique_convergence_2022}. More recently, \cite{davini_convergencedivergence_2024} studied the general convergence/divergence behavior of solutions in the vanishing discount process without assuming monotonicity in the \( u \)-variable, building on the variational framework developed in their earlier works \cite{jin_nonlinear_2023,
wang_aubrymather_2019, wang_variational_2019CMP}.
The bounded domain case with state-constraint conditions (which we employed some results in this paper) was initiated in \cite{tu_generalized_2024}. 
\medskip

Research on noncompact domains remains relatively limited. The weak KAM theory for classical Hamiltonians on noncompact domains was first developed in \cite{fathi_weak_2007}. For contact Hamilton-Jacobi equations on noncompact manifolds, \cite{wang_variational_2019} investigated the finite horizon problem with specific initial conditions, while \cite{fleming_controlled_2006} approached similar questions from optimal control theory. The noncompact domain were treated in \cite{capuzzo_dolcetta_vanishing_2023,
ishii_vanishing_2020}.
Our work is particularly motivated by \cite{ishii_vanishing_2020}, which examined the vanishing discount problem in the discounted case.

\medskip

For a treatment of the contact case using the nonlinear adjoint method, we refer to \cite{jing_generalized_2020}. Besides, Mather measures in nonconvex settings are studied in \cite{cagnetti_aubrymather_2011,
fathi_pde_2005}, and more recently in \cite{tran_representation_2025}, though many questions remain open.

\subsection{Organization of the paper} 

This paper is organized as follows.
Section \ref{sec:notation} introduces notation, terminology, and preliminaries for the ergodic problem \eqref{eq:E}, covering the well-posedness of $c(H)$ and the Aubry set. Section \ref{sec:state-constraint} recalls the variational formula for the state-constraint problem, discusses an exponential form for the contact case, and presents the vanishing discount result on bounded domain.
Section \ref{sec:MaximalSolution} constructs the maximal solution \( u_\lambda \in C(\R^n) \) for \( \lambda > 0 \) via the Perron method.
Its variational representation is established in Section \ref{sec:variational} under additional conditions.
Section \ref{sec:selection-principle} is dedicated to the proof of Theorem \ref{thm:selection-principle}.
Section \ref{sec:localization} proves Theorem \ref{thm:localization} and demonstrates an alternative proof of Theorem \ref{thm:selection-principle} as a consequence (see Subsection \ref{subsection:alternative-thm}). 
The Appendix contains properties of Lagrangians, the comparison principle, and other technical results.

\section{Notations, Definitions, and Preliminaries for the Ergodic Problem} \label{sec:notation}
For $U\subset\R^n$, we denote by $\mathrm{AC}([0,t];U)$ the set of all absolutely continuous functions $f:[0,t]\to U$. For $(x,t) \in  \R^n\times [0,\infty)$, we define 
\begin{align*}
	\mathcal{C}^+(x;t;U) = \{\gamma\in \mathrm{AC}([0,t];U): \gamma(t)=x\} 
	\quad\text{and}\quad 
	\mathcal{C}^-(x;t;U) = \{\gamma\in \mathrm{AC}([-t,0];U): \gamma(0)=x\}. 
\end{align*}
For a function $F(x,p,u):U\times \R^n\times \R \to \R$, we denote 
\begin{equation*}
	\partial^+_uF(x,p,u) = \lim_{ h \to 0^+}\frac{F(x,p,u + h) - F(x,p,u)}{h}. 
\end{equation*}
If $u\mapsto F(x,p,u)$ is monotone, this limit always exists. \medskip

We denote by $\mathcal{P}(U)$ and $\mathcal{R}(U)$ the set of probability and Radon measures on $U$, respectively. For a sequence of Radon measures $\mu_k \in \mathcal{R}(U)$, we say $\mu_n\rightharpoonup \mu$ in the weak$^*$ topology if $\int_ U f\;d\mu_n \to \int_U f\;d\mu$ for all $f\in C_c(U)$. When \( U = \mathbb{R}^n \times \mathbb{R}^n \), we write \( \langle \mu, f \rangle := \int_{\mathbb{R}^n \times \mathbb{R}^n} f \, d\mu \). \medskip

A continuous function \( u:\R^n \to \R \) is a \emph{viscosity subsolution} (resp., \emph{supersolution}) of \eqref{eq:DP} if for every \( \varphi \in C^1(\R^n) \), the inequality \( H(x, D\varphi(x), \lambda u(x)) \leq 0 \) (resp., \( \geq 0 \)) holds at any point \( x \) where \( u - \varphi \) attains a local maximum (resp., minimum). It is a \emph{viscosity solution} if it satisfies both conditions. \medskip

\subsection{The Legendre transform} 
Under \ref{itm:assumptions-h1}, we define the Legendre transform of $H$ as:
\begin{align}\label{eq:Legendre}
	L(x,v,u) = \sup_{p\in \R^n} \Big(p\cdot v - H(x,p,u)\Big), \qquad (x,v,u)\in \R^n\times \R^n\times \R. 
\end{align}
We note that, under 
\ref{itm:assumptions-h2}, 
there exists a constant $c_0$ such that $\max_{x\in \R^n} H(x,0, 0) \leq c_0$. Therefore
\begin{align}\label{eq:BoundL}
	L(x,v, 0) \geq -H(x,0,0) \geq -\max_{x\in \R^n} H(x,0, 0)  \geq -c_0\qquad\text{for all}\;(x,v)\in \R^n\times\R^n. 
\end{align}
Under \ref{itm:assumptions-h1}, it is standard that for every fixed $u\in \R$, $x\mapsto L(x,v,u)$ is convex, and that for every compact set $K\subset \R^n$ we have
\begin{equation}\label{eq:SuperLinearLtilde}
	\lim_{|v|\to \infty} \min_{x\in K} \frac{L(x,v,u)}{|v|} = +\infty.
\end{equation}
By 
\ref{itm:assumptions-h2}, 
there exist \( \delta_0 > 0 \) and \( R_0 > 0 \) such that  $\max_{|x|\geq R_0} \left\lbrace H(x,p,0): |p|\leq \delta_0 \right\rbrace < c(H)$. 
Hence
\begin{equation}\label{eq:growthL}
	\min_{|x|\geq R_0} L(x,v, 0) 
	\geq \delta |v| - 
	\max_{|x|\geq R_0} 
	\left\lbrace H(x,p,0): |p|\leq \delta_0 \right\rbrace > \delta_0 |v| - c(H). 
\end{equation}
Also, by 
\ref{itm:assumptions-h1}, 
we can also choose $R_0>0$ such that 
\begin{equation}\label{eq:BoundLPositive}
	\limsup_{|x|\to \infty} H(x,0,0) < c(H) 
	\qquad\Longrightarrow\qquad 
	\min_{|x|\geq R_0} L(x,v, 0) + c(H) \geq \delta_0 > 0.
\end{equation}
These estimates are analogous to those in the discounted case studied in \cite{ishii_vanishing_2020}.

\subsection{The critical value, the ergodic problem and the Aubry set} \label{sec:cHAubry}
We show that the ergodic constant \( c(H) \), defined in \eqref{eq:c(H)}, is well-defined and finite.

\begin{lem}\label{lem:critical-value} Assume 
\ref{itm:assumptions-h1}, \ref{itm:assumptions-h2}. 
Then \( c(H) \), as defined by \eqref{eq:c(H)}, is finite and, in fact, a minimum with 
    \begin{equation*}
        c(H) \geq m_0=\max_{x\in \R^n} \min_{p\in \R^n} H(x,p,0).      
    \end{equation*}
\end{lem}

We note that Lemma \ref{lem:critical-value} in fact holds without assuming the convexity of \( H \).
The following lemma relies on the convexity assumption in \ref{itm:assumptions-h1} (see \cite[Corollary 2.36]{tran_hamilton-jacobi_2021}).

\begin{lem}[{\cite[Proposition 2.1]{ishii_vanishing_2020}}]\label{lem:cutoffIshiiSiconofli}
Assume 
\ref{itm:assumptions-h1}, \ref{itm:assumptions-h2}. 
Let \( u \in C(\mathbb{R}^n) \) be a subsolution to \( H(x, Du, 0) = a \) in \( \mathbb{R}^n \), and let \( K \subset \mathbb{R}^n \) be a compact set.
\begin{itemize}
    \item[(i)] There exists a subsolution \( \tilde{v} \in C(\mathbb{R}^n) \) to \( H(x, D\tilde{v}, 0) = a \) in \( \mathbb{R}^n \) such that \( \tilde{v}(x) = -\frac{\varepsilon}{2} |x| \) for \( |x| \geq R \), for some sufficiently large \( R \).
    \item[(ii)] There exists $\tilde{u}\in C(\R^n)$ a subsolution to $H(x,D\tilde{u},0) = a$ in $\R^n$ such that $\title{u} = u$ in $K$ and $\tilde{u} \equiv \text{\rm constant}$ in $\R^n\backslash B_R(0)$ for some $R$ large enough.
\end{itemize}
\end{lem}

\begin{defn} We define the intrinsic semi distance $S_H(\cdot, \cdot)$ in $\R^n$ by
\begin{equation*}
    S_H(x,y) = \sup \left \lbrace u(x)-u(y): u\;\text{is a subsolution to}\;H(x,Du,0)=c(H)\;\text{in}\;\R^n
    \right\rbrace, \qquad x,y\in \R^n.
\end{equation*}     
\end{defn}

\begin{lem}\label{lem:mane-potential} Assume 
\ref{itm:assumptions-h1}, \ref{itm:assumptions-h2}. 
\begin{itemize}
    \item[(i)] For each fixed $y_0\in \R^n$, $x\mapsto S_H(x,y_0)$ is well-defined and is locally Lipschitz continuous on $\R^n$. Moreover, $S_H(x,x) = 0$ for all $x\in \R^n$, and $x\mapsto S_H(x,y_0)$ is a subsolution to $H(x,Du, 0) = c(H)$ in $\R^n$. 
    \item[(ii)] For $x,y,z\in \R^n$, we have $S_H(x,y) \leq S_H(x,z)+S_H(z,y)$. 
    \item[(iii)] The map $x\mapsto S_H(x,y_0)$ is a viscosity solution to $H(x,Du,0) = c(H)$ in $\R^n\backslash \{y_0\}$. 
\end{itemize}
\end{lem}

\begin{proof} The first part is clear from 
    \eqref{eq:assumptions-LocalCoercive} in \ref{itm:assumptions-h1},
    and it is also clear from the definition that $S_H(x,x) = 0$. The fact that $x\mapsto S_H(y)$ is a subsolution follows from stability of viscosity subsolution. \medskip

    For the second part, let us define $w(x) = S_H(x,y) - S_H(z,y)$ for $x\in \R^n$, we observe that $w(z) = 0 = S_H(z,z)$. For $y,z$ $\in \R^n$, $w$ is a subsolution to $H(x,Dw,0)=c(H)$ in $\R^n$, thus by definition of $S_H$ we have
    \begin{equation*}
        w(x) = w(x) - w(z) \leq S_H(x,z) \qquad\Longrightarrow\qquad S_H(x,y) - S_H(z,y) \leq S_H(x,z). 
    \end{equation*}
    Finally, for the last part, the function \( w(x) = S_H(x, y_0) \) is a subsolution to \( H(x, Du, 0) = c(H) \) in \( \mathbb{R}^n \) and is locally Lipschitz in \( \mathbb{R}^n \). By virtue of Perron's method, since \( w \) is the maximal subsolution with \( w(y_0) = 0 \), we conclude that \( w \) is a viscosity solution to \( H(x, Du, 0) = c(H) \) in \( \mathbb{R}^n \setminus \{y_0\} \).
\end{proof}

\begin{defn}[The Aubry set]\label{defn:AubrySet} Under 
\ref{itm:assumptions-h1}, \ref{itm:assumptions-h2}, 
the Aubry set of $H$ in $\R^n$ is defined as follows:
\begin{equation*}
    \mathcal{A} = \{y\in \R^n: x\mapsto S_H(x,y)\;\text{is a solution to $H(x,Du,0) = c(H)$ in $\R^n$}\}. 
\end{equation*}
\end{defn}

\begin{lem}[{\cite[Proposition A.2]{ishii_vanishing_2020}}]\label{lem:strict-sub} 
Assume 
\ref{itm:assumptions-h1}, \ref{itm:assumptions-h2}. 
Given \( y \in \mathbb{R}^n \), if there exists a subsolution to \( H(x, Du, 0) = c(H) \) in \( \mathbb{R}^n \) that is strict in a neighborhood of \( y \), then \( y \notin \mathcal{A} \). Conversely, if \( y \notin \mathcal{A} \), there exists a subsolution to \( H(x, Du, 0) = c(H) \) in \( \mathbb{R}^n \) that is strict in a neighborhood of \( y \).  
\end{lem}

\begin{lem}[{\cite[Proposition 4.3]{ishii_vanishing_2020}}]\label{lem:AubrySetNonempty} Assume 
\ref{itm:assumptions-h1}, \ref{itm:assumptions-h2}. 
Then there exists a viscosity solution \( u \in C(\mathbb{R}^n) \) to \( H(x,Du,0) = c(H) \) in \( \mathbb{R}^n \); equivalently, the Aubry set \( \mathcal{A} \) is a nonempty compact subset of $\R^n$.
\end{lem}

\section{The state-constraint problem} \label{sec:state-constraint}
Throughout this subsection, assume \( \Omega \subset \R^n \) is open, bounded, connected, with \( C^2 \) boundary.  
We recall the variational formula for \( \vartheta_\lambda \), introduce discounted and Mather measures, and extend \cite{tu_generalized_2024} with a new exponential formula. These show \( \vartheta_\lambda \to \vartheta \) as \( \lambda \to 0^+ \), where both solve the corresponding state-constraint problems: 
\begin{equation}\label{eq:propStateConstraintBoundOmega}
    \begin{cases}
        \begin{aligned}
            H\left(x,D\vartheta_\lambda , \lambda \vartheta_\lambda \right) &\leq c_\Omega(H) && \text{in}\;\Omega, \\
            H\left(x,D\vartheta_\lambda , \lambda  \vartheta_\lambda \right) &\geq c_\Omega(H) && \text{on}\;\overline{\Omega},
        \end{aligned}
    \end{cases}
\end{equation}
and
\begin{equation}\label{eq:ErgodicStateOmega}
    \begin{cases}
        \begin{aligned}
            H(x,D\vartheta, 0) &\leq c_\Omega(H) &&\text{in}\;\Omega, \\
            H(x,D\vartheta, 0) &\geq c_\Omega(H) &&\text{on}\;\overline{\Omega},         
        \end{aligned}
    \end{cases}
\end{equation}
respectively. 
Under
\ref{itm:assumptions-h1} and \ref{itm:assumptions-h3}, 
well-posedness for \eqref{eq:propStateConstraintBoundOmega} is known (see \ref{itm:assumptions-O1} in Appendix,  \cite{capuzzo-dolcetta_hamilton-jacobi_1990,mitake_asymptotic_2008,soner_optimal_1986, tu_generalized_2024}.
 \medskip

\subsection{Solution to the state-constraint problem} 

\begin{defn}\label{defn:ergodicStateConstraint} Assume 
\ref{itm:assumptions-h1}. 
We define the ergodic constant on \( \Omega \) as
\begin{equation}\label{eq:ergodicOmegaH}
    c_\Omega(H) = \inf \{c\in \R^n: H(x,Du, 0) = c\;\text{in}\;\Omega\;\text{admits a viscosity subsolution}\;u\in C(\Omega)\}. 
\end{equation}
\end{defn}

\begin{prop}\label{prop:StateConstraintLinear} 
Assume 
\ref{itm:assumptions-h1}.
The constant \( c_\Omega(H) \), as defined in \eqref{eq:ergodicOmegaH}, is well-defined and attains a minimum. For each $\lambda\in (0,1)$, there exists a unique state-constraint viscosity solution $v_\lambda\in \mathrm{Lip}(\Omega)$ to 
\begin{equation}\label{eq:discountedStateOmega}
    \begin{cases}
        \begin{aligned}
            \lambda v_\lambda + H(x,Dv_\lambda, 0) &\leq 0 &&\text{in}\;\Omega, \\
            \lambda v_\lambda + H(x,Dv_\lambda, 0) &\geq 0 &&\text{on}\;\overline{\Omega}
        \end{aligned}
    \end{cases}
\end{equation}
with a priori estimate $\Vert\lambda v_\lambda\Vert_{L^\infty(\Omega)} + \Vert Dv_\lambda\Vert_{L^\infty(\Omega)} \leq C_\Omega$ in dependent of $\lambda$. Furthermore, 
\begin{equation*}
    \lambda v_\lambda \to -c_\Omega(H)\qquad\text{uniformly on}\;\overline{\Omega}. 
\end{equation*}
There exists \( v \in C(\Omega) \) that solves \eqref{eq:ErgodicStateOmega}, and \( c_\Omega(H) \) is the unique constant for which \eqref{eq:ErgodicStateOmega} admits a continuous viscosity solution. Moreover, \( \| v \|_{L^\infty(\Omega)} \leq C_\Omega \), where \( C_\Omega \) depends on \( \Omega \) and \( H \), as specified in \ref{itm:assumptions-h1}.
\end{prop}

\begin{prop}\label{prop:UniformBoundedLocally} Assume 
\ref{itm:assumptions-h1}, \ref{itm:assumptions-h3}. 
Let \( \vartheta_\lambda \in C(\Omega)\) be the solution to the state-constraint problem \eqref{eq:propStateConstraintBoundOmega}. There exists a constant \( C_\Omega \) depending on \( \Omega \) and \( H \), as per assumptions 
\ref{itm:assumptions-h1} and \ref{itm:assumptions-h3} such that 
\begin{equation*}
	\| \vartheta_\lambda  \|_{L^\infty(\Omega)}  + \Vert D\vartheta_\lambda  \Vert_{L^\infty(\Omega)} \leq C_\Omega.
\end{equation*}
\end{prop}
\begin{proof} Let $v\in C(\overline{\Omega})$ be a viscosity solution to \eqref{eq:ErgodicStateOmega}, as in Proposition \ref{prop:StateConstraintLinear}. We define $v^{\pm} = v \pm \Vert v\Vert_{L^\infty(\Omega)}$. 
By assumption \ref{itm:assumptions-h3}, we have 
\begin{equation*}
\begin{cases}
    H(x, Dv^-, \lambda v^-) \leq H(x, Dv^-, 0) \leq c_\Omega(H) \qquad\text{in}\;\Omega, \\ 
    H(x, Dv^+, \lambda v^+) \geq H(x, Dv^+, 0) \geq c_\Omega(H) \qquad\text{on}\;\overline{\Omega}.
\end{cases}
\end{equation*}
By the comparison principle (see, for instance, \cite[Theorem 2.3]{tu_generalized_2024}), we have $v^- \leq \vartheta_\lambda \leq v^+ $ in $\Omega$.
Thus, \( \| \vartheta_\lambda \|_{L^\infty(\Omega)} \leq C_\Omega \). Then from \eqref{eq:assumptions-LocalCoercive} in \ref{itm:assumptions-h1} and the fact that $\Omega$ is bounded, by increasing the constant $C_\Omega$ if necessary, we deduce that $\Vert D\vartheta_\lambda\Vert_{L^\infty(\Omega)} \leq C_\Omega$.
\end{proof}

\begin{prop} Assume 
\ref{itm:assumptions-h1}, \ref{itm:assumptions-h3}. 
The state constraint solution $\vartheta_\lambda$ to \eqref{eq:propStateConstraintBoundOmega} satisfies: 
\begin{align}\label{eq:StateConstraintOptimalControl}
	\vartheta_\lambda(x)  = \inf_{\gamma\in \mathcal{C}^-(x;t;\overline{\Omega})} 
	\left\lbrace 
	\vartheta_\lambda(\gamma(-t)) 
	+ 
	\int_{-t}^0 
	\Big(L\big(\gamma(s), \dot{\gamma}(s), \lambda \vartheta_\lambda(\gamma(s))\big) + c(H)\Big)\;ds 
	\right\rbrace. 
\end{align}
\end{prop}

\begin{proof} Given that \( \vartheta_\lambda\in C(\overline{\Omega}) \) solves \eqref{eq:propStateConstraintBoundOmega}, we define $\widetilde{H}_\Omega(x,p) = H(x,p, \lambda \vartheta_\lambda(x))$ for $(x,p) \in \R^n\times \R^n$. Since $\vartheta_\lambda \in C(\overline{\Omega})\cap \mathrm{Lip}(\overline{\Omega})$, we can verify that $\widetilde{H}_\Omega$ satisfies standard assumptions on convexity, coercivity and continuity as in \cite{mitake_asymptotic_2008}. Using the results from \cite[Theorem 3.5, Theorem 5.8]{mitake_asymptotic_2008}, the time-dependent problem
\begin{equation}\label{eq:CauchyStateConstraint}
	\begin{cases}
	\begin{aligned}
		v_t + H(x,Dv,\lambda \vartheta_\lambda) &= 0 && \text{in}\;\Omega\times (0,\infty), \\
		v_t + H(x,Dv,\lambda \vartheta_\lambda) &\geq 0 && \text{on}\;\overline{\Omega}\times (0,\infty), \\
		v(x,0) &= \vartheta_\lambda(x) && \text{on}\;\overline{\Omega},
	\end{aligned}
	\end{cases}
\end{equation}
admits a unique viscosity solution $v \in C\left(\overline{\Omega}\times [0,\infty)\right)\cap W^{1,\infty}\left(\overline{\Omega}\times (0,\infty)\right)$, and it is given by
\begin{align}\label{eq:MinimizerStateConstraint}
	v(x,t) = \inf_{\gamma\in \mathcal{C}^-(x;t;\overline{\Omega})} 
	\left\lbrace 
	\vartheta_\lambda(\gamma(-t)) 
	+ 
	\int_{-t}^0 
	L\Big(\gamma(s), \dot{\gamma}(s), \lambda \vartheta_\lambda(\gamma(s))\Big)\;ds 
	\right\rbrace
\end{align}
for $(x,t)\in \overline{\Omega}\times [0,\infty)$. By the comparison principle \cite[Theorem 3.5]{mitake_asymptotic_2008} we obtain $v(x,t) = \vartheta_\lambda(x) - c(H)t$ for $(x,t)\in \overline{\Omega}\times [0,\infty)$. As a result we obtain the conclusion \eqref{eq:StateConstraintOptimalControl}.
\end{proof}

\begin{prop}\label{prop:ExistenceOfMinimizerTimeOmega} 
Assume \ref{itm:assumptions-h1}, \ref{itm:assumptions-h3}. 
\begin{itemize}
\item[$\mathrm{(i)}$] For $(x,t)\in \overline{\Omega}\times \R^n$, there exists a minimizer $\gamma\in \mathcal{C}^-(x,t;\overline{\Omega})$ to  \eqref{eq:StateConstraintOptimalControl}. 
\item[$\mathrm{(ii)}$] There exists $\widehat{C}_\Omega$ such that any miminimizer $\gamma\in \mathcal{C}^-(x,t;\overline{\Omega})$ to \eqref{eq:StateConstraintOptimalControl} satisfies $|\dot{\gamma}(s)| \leq \widehat{C}_\Omega$ for a.e. $s\in (-t,0)$.
\item[$\mathrm{(iii)}$] There exists $\gamma\in \mathrm{AC}((-\infty,0];\overline{\Omega})$ with $\gamma(0) = x$ such that 
\begin{align}\label{eq:ExtrtemalInftyTime}
	\vartheta_\lambda(\gamma(a)) - \vartheta_\lambda(\gamma(b)) = \int_a^b \Big(L\big(\gamma(s),\dot{\gamma}(s), \lambda \vartheta_\lambda(\gamma(s))\big)+c_\Omega(H)\Big)\;ds
\end{align}
for every $-\infty < b<a<0$, and $|\dot{\gamma}(s)| \leq \widehat{C}_\Omega$ for a.e. $s\in (-\infty,0]$.
\end{itemize}
\end{prop}

\begin{proof} The existence of a minimizer to \eqref{eq:MinimizerStateConstraint} follows from standard arguments, using that \( \vartheta_\lambda \in \mathrm{BUC}(\overline{\Omega}) \cap \mathrm{Lip}(\overline{\Omega}) \); see \cite{mitake_asymptotic_2008} for the role of lower semicontinuity in the proof. This minimizer also solves \eqref{eq:StateConstraintOptimalControl}, as desired.\medskip

If $\gamma\in \mathcal{C}^-(x,t;\overline{\Omega})$ is a minimizer to \eqref{eq:StateConstraintOptimalControl}, then it satisfies \eqref{eq:ExtrtemalInftyTime} by the standard Dynamic Programming Principle. 
We compute
\begin{align*}
	\vartheta_\lambda(\gamma(a)) - \vartheta_\lambda(\gamma(b)) 
	&= \int_{b}^a \Big(L(\gamma(s), \dot{\gamma}(s), \lambda \vartheta_\lambda (\gamma(s))) + c(H)\Big)\;ds
	\geq \int_{b}^a \Big(L(\gamma(s), \dot{\gamma}(s), C_\Omega) + c(H)\Big)\;ds
\end{align*}
for $-t<b<a<0$, where we use the fact that $u\mapsto L(x,v,u)$ is nonincreasing, and $\Vert \lambda \vartheta_\lambda \Vert_{L^\infty(\Omega)} \leq C_\Omega$ by Proposition \ref{prop:UniformBoundedLocally}. Using the fact that $|D\vartheta_\lambda|\leq C_\Omega$, we have
\begin{align*}
	\frac{1}{a-b}\int_{b}^a \Big(L(\gamma(s), \dot{\gamma}(s), C_\Omega) + c(H)\Big)\;ds \leq C_\Omega \cdot \frac{|\gamma(a) - \gamma(b)|}{a-b}. 
\end{align*}
Thus, for a.e. $s\in (-t,0)$, where $\gamma$ is differentiable, we have
\begin{align*}
	L(\gamma(s), \dot{\gamma}(s), C_\Omega) +c(H)\leq C_\Omega |\dot{\gamma}(s)| \qquad\text{for a.e.}\;s\in (-t,0). 
\end{align*}
Therefore
\begin{align*}
	\left(\frac{L(\gamma(s), \dot{\gamma}(s), C_\Omega)}{|\dot{\gamma}(s)|} -C_\Omega\right) \cdot| \dot{\gamma}(s)| + c(H) \leq 0. 
\end{align*}
By \eqref{eq:assumptions-LocalCoercive} in \ref{itm:assumptions-h1} 
and the fact that $\gamma(s)\in \overline{\Omega}$ is bounded, we deduce that $|\dot{\gamma}(s)| \leq \widehat{C}_\Omega$ for a.e. $s\in (-t,0)$.  \medskip

For (ii), a simple way to prove the existence of such a curve $\gamma\in \mathcal{C}^-(x,\infty;\overline{\Omega})$ is to connect curves on smaller intervals, similar to \cite[Corollary 6.2]{ishii_asymptotic_2008}. We say $\gamma\in \mathrm{AC}([-t,0];\overline{\Omega})$ is an extremal curve for $u_{\lambda,\Omega}$ on $[-t,0]$ if 
\begin{equation*}
	\vartheta_\lambda(\gamma(b)) - \vartheta_\lambda(\gamma(a)) = \int_{a}^b \Big(L(\gamma(s), \dot{\gamma}(s), \lambda \vartheta_\lambda(s))+c(H)\Big)\;ds \qquad\text{for all}\; -t<b<a<0. 
\end{equation*}
We denote by $\mathcal{E}([b,a])$ the set of all extremal curve $\gamma:[b,a]\to \overline{\Omega}$ for $u_{\lambda, \Omega}$ on $[b,a]$. For any $y\in \overline{\Omega}$, we can find $\gamma_y\in \mathcal{E}([-1,0])$ with $\gamma_y(0) = y$. 
We define the sequence of minimizer $\{\xi_j\}_{j=1}^\infty \subset \mathcal{E}([-1,0])$ by $\xi_1 = \gamma_x$, $\xi_2 = \gamma_{\xi_1(-1)}$, $\xi_3 = \gamma_{\xi_2(-1)}$, \ldots, and
\begin{equation*}
\gamma(s) = \begin{cases}
\begin{aligned}
	&\xi_1(s)   && s\in [-1,0], \\
	&\xi_2(s+1) && s\in [-2,-1], \\
	&\xi_3(s+2) && s\in [-3,-2], \\
	& \ldots 
\end{aligned}
\end{cases}
\end{equation*}
Then we obtain $\gamma\in \mathcal{E}(-\infty, 0])$ with $\gamma(0) = x$. By the construction $\gamma$ is the curve that satisfies \eqref{eq:ExtrtemalInftyTime}. Another method to construct a curve $\gamma$ satisfying \eqref{eq:ExtrtemalInftyTime} is to apply a diagonal argument to a sequence of extremal curves $\gamma_k \in \mathcal{E}([-k, 0])$ with $\gamma_k:[-k,0]\to\overline{\Omega}$ such that $\gamma_k(0) = x$. The boundedness of $\overline{\Omega}$ allows us to apply a diagonal argument to obtain the conclusion.
\end{proof}

Next, we transform \eqref{eq:StateConstraintOptimalControl} into an exponential form. 

\begin{defn}\label{defn:DiscountedIndexOmega} 
Assume \ref{itm:assumptions-h1}, \ref{itm:assumptions-h3}. 
Let $\vartheta_\lambda$ be the solution to \eqref{eq:propStateConstraintBoundOmega}. For $(x,v)\in \R^n\times\R^n$, we define
\begin{equation}\label{eq:DiscountedIndexOmega}
    \kappa_{\lambda}(x,v) = 
    \begin{cases}
    \partial_u^+ L(x,v,0) &\text{if}\; \vartheta_\lambda(x) = 0, \vspace{0.1cm}\\ 
    \dfrac{L(x,v,\lambda \vartheta_\lambda(x) - L(x,v,0)}{\lambda \vartheta_\lambda(x)} &\text{if}\;\vartheta_\lambda(x) \neq 0,
    \end{cases}
\end{equation}
where $\partial_u^+L(x,p,u)$ is the right-derivative with respect to $u$, which exists due to the monotonicity of $u\mapsto L(x,p,u)$. 
\end{defn}

We summarize the basic properties of \( \kappa_\lambda \) below; the proof is straightforward and omitted.

\begin{lem} 
Assume \ref{itm:assumptions-h1}, \ref{itm:assumptions-h3}. 
\begin{itemize}

\item[$\mathrm{(i)}$] $\kappa_\lambda(x,v)\leq 0$ for all $(x,v)\in \overline{\Omega}\times\R^n$. If $u\mapsto L(x,v,u)$ is differentiable at $0$ for a.e. $(x,v)$, then $\kappa_\lambda(x,v) \to \partial_uL(x,v,0)$ as $\lambda\to 0^+$ for a.e. $(x,v)\in \overline{\Omega}\times \R^n$.

\item[$\mathrm{(ii)}$] Assume further that, for any compact sets $K\subset \R^n$ and $I\subset \R$ then
\begin{align} \label{eq:superlinearOmega}
	\lim_{|p|\to \infty} \min_{x\in K} \left\lbrace \frac{H(x,p,u)}{|p|}: u\in I  \right\rbrace= +\infty. 
\end{align}
Furthermore, for $|v|\leq C$ then $-\overline{\kappa}_C \leq \kappa_\lambda(x,v)\leq -\underline{\kappa}_C$ for all $|v|\leq C$.
\end{itemize}
\end{lem}

\begin{rmk} For studying the maximal solution \( u_\lambda \) to \eqref{eq:DP}, we can assume without loss of generality that \eqref{eq:superlinearOmega} holds, as shown in \emph{Proposition \ref{prop:SuperlinearReduction}}. Under the additional assumption \eqref{eq:superlinearOmega}, it is well-known that if \( |v| \leq C \), the supremum in the definition of the Legendre transform \eqref{eq:Legendre} can be restricted to \( |p| \leq \tilde{C} \) for some \( \tilde{C} > 0 \).
\end{rmk}

\begin{prop}\label{prop:StateConstraintProperties} 
Assume \ref{itm:assumptions-h1}, \ref{itm:assumptions-h3}. 
Let $\vartheta_\lambda$ be the solution to \eqref{eq:propStateConstraintBoundOmega}. For $(x,t)\in \overline{\Omega}\times (0,\infty)$, we have 
\begin{align}\label{eq:OptimalControlStateConstraintExp}
	\vartheta_\lambda(x) = \inf_{\gamma\in \mathcal{C}^-(x,t;\overline{\Omega})} 
	\left\lbrace 
	e^{\lambda \alpha^\lambda_\gamma(-t) }\vartheta_\lambda(\gamma(-t))  + \int_{-t}^0 e^{\lambda\alpha_{\gamma}^\lambda (s)} \Big(L(\gamma(s), \dot{\gamma}(s), 0) + c_\Omega(H)\Big)\;ds
	\right\rbrace,
\end{align}
where 
\begin{equation*}
	\alpha_{\gamma}^\lambda(s) = \int_{s}^0 \kappa_{\lambda}(\gamma(\tau), \dot{\gamma}(\tau))\;d\tau \qquad 
	\text{for}\;s\in [-t,0]. 
\end{equation*}

\begin{itemize}
\item[$\mathrm{(i)}$] For $(x,t)\in \overline{\Omega}\times (0,\infty)$, there exists a minimizer $\gamma\in \mathcal{C}^-(x,t;\overline{\Omega})$ to  \eqref{eq:OptimalControlStateConstraintExp}. 

\item[$\mathrm{(ii)}$] There exists $\widehat{C}_\Omega$ such that any miminimizer to \eqref{eq:OptimalControlStateConstraintExp} satisfies $|\dot{\gamma}(s)| \leq \widehat{C}_\Omega$ for a.e. $s\in (-t,0)$.

\item[$\mathrm{(iii)}$] There exists $\gamma\in \mathrm{AC}((-\infty,0];\overline{\Omega})$ with $\gamma(0) = x$ such that 
\begin{align}\label{eq:ExtrtemalInftyExp}
	e^{\lambda \alpha^{\lambda}_\gamma(a)}	
	\vartheta_\lambda(\gamma(a)) 
	- 
	e^{\lambda \alpha^{\lambda}_\gamma(b)}
	\vartheta_\lambda(\gamma(b)) 
	= \int_a^b e^{\lambda \alpha^\lambda_\gamma(s)}\Big(L\big(\gamma(s),\dot{\gamma}(s), 0)\big)+c_\Omega(H)\Big)\;ds
\end{align}
for every $-\infty < b<a<0$, and $|\dot{\gamma}(s)| \leq \widehat{C}_\Omega$ for a.e. $s\in (-\infty,0]$.
\end{itemize}
Furthermore, there exists a minimizer to \eqref{eq:OptimalControlStateConstraintExp}, and a constant $\widehat{C}_\Omega$ such that any miminimizer to \eqref{eq:OptimalControlStateConstraintExp} satisfies 
\begin{align*}
	|\dot{\gamma}(s)| \leq \widehat{C}_\Omega \qquad\text{for a.e.}\; s \in (-t,0).
\end{align*}
\end{prop}

\begin{proof} For any $\gamma\in \mathcal{C}^-(x,t;\overline{\Omega})$ and $-t<b<a<0$, by Proposition \ref{prop:curve} we have
\begin{equation}\label{eq:estab}
	\vartheta_\lambda(\gamma(a)) \leq \vartheta_\lambda(\gamma(b)) +  \int_{b}^a\Big( L\left((\gamma(\tau), \dot{\gamma}(\tau), \lambda \vartheta_\lambda(\gamma(\tau))\right) + c_\Omega(H)\Big)\;d\tau.
\end{equation}
By \cite[Proposition 4.1]{mitake_asymptotic_2008}, 
$s\mapsto \vartheta_\lambda(\gamma(s))$ is absolutely continuous, thus from \eqref{eq:estab} we have 
\begin{align*}
	\frac{d}{ds}\big(\vartheta_\lambda(\gamma(s))\big) 
	\leq L(\gamma(s), \dot{\gamma}(s), \lambda \vartheta_\lambda (\gamma(s))) + c(H) \qquad\text{for a.e.}\;s\in (-t,0). 
\end{align*}
For a.e. $s\in (-t,0)$ we have
\begin{align*}
	\frac{d}{ds}\left(e^{\lambda \alpha_{\gamma}^\lambda(s)}  \vartheta_\lambda (\gamma(s))   \right) 
	&= e^{\lambda \alpha_{\gamma}^\lambda(s)}
	\left(
	-\lambda \vartheta_\lambda(\gamma(s)) \cdot \kappa_{\lambda}(\gamma(s), \dot{\gamma}(s))  
	+ 
	\frac{d}{ds}\big(\vartheta_\lambda(\gamma(s))\big) 
	\right) \\
	&\leq e^{\lambda \alpha_{\gamma}^\lambda (s)}\Big(-\lambda \vartheta_\lambda (\gamma(s)) \cdot \kappa_{\lambda }(\gamma(s), \dot{\gamma}(s))  + L(\gamma(s), \dot{\gamma}(s), \lambda \vartheta_\lambda(\gamma(s))) + c_\Omega(H)\Big) \\
	&=e^{\lambda \alpha_{\gamma}^\lambda(s)}\Big( L(\gamma(s), \dot{\gamma}(s), 0) +c_\Omega(H)\Big).
\end{align*}
Integrating over $[-t,0]$, we obtain
\begin{align}\label{eq:estimateEXPAC}
	e^{\lambda \alpha_{\gamma}^\lambda(0)}\vartheta_\lambda(\gamma(0)) 
	\leq 
	e^{\lambda \alpha_{\gamma}^\lambda(-t)}\vartheta_\lambda(\gamma(-t))  
	+ 
	\int_{-t}^0 e^{\lambda\alpha_{\gamma}^\lambda(s)} 
	\Big(L(\gamma(s), \dot{\gamma}(s), 0) + c_\Omega(H)\Big)\;ds. 
\end{align}
Equivalently, we have shown that 
\begin{align*}
	\vartheta_\lambda(x) 
	\leq \inf_{\gamma\in \mathcal{C}^-(x,t;\overline{\Omega})} \left\lbrace
	e^{\lambda \alpha_{\gamma}^\lambda(-t)}\vartheta_\lambda(\gamma(-t))  + \int_{-t}^0 e^{\lambda\alpha_{\gamma}^\lambda(s)} \Big(L(\gamma(s), \dot{\gamma}(s), 0) + c_\Omega(H) \Big)\;ds \right\rbrace.
\end{align*}
By Proposition \ref{prop:ExistenceOfMinimizerTimeOmega}, let \( \gamma \in \mathcal{C}^-(x, t; \overline{\Omega}) \) be a minimizer of \eqref{eq:StateConstraintOptimalControl}. Repeating the argument above starting from \eqref{eq:estab}, now with equality, we conclude that \( \gamma \) also minimizes \eqref{eq:OptimalControlStateConstraintExp} and satisfies the velocity bound \( |\dot{\gamma}(s)| \leq \widehat{C}_\Omega \) from Proposition \ref{prop:ExistenceOfMinimizerTimeOmega}.
\end{proof}

We observe that \ref{itm:assumptions-h3} implies both \eqref{eq:MonotoneWeakBelowH} and \eqref{eq:MonotoneWeakAboveH}.
\begin{align}
    &u\mapsto H(x,p,u) - \underline{\kappa}_C\cdot u \;\text{is nondecreasing for}\; |p|\leq C,  \label{eq:MonotoneWeakBelowH} \\
    &u\mapsto H(x,p,u) - \overline{\kappa}_C\cdot u \;\text{is nonincreasing for}\; |p|\leq C. \label{eq:MonotoneWeakAboveH} 
\end{align}
In the next proposition, we require only \eqref{eq:MonotoneWeakBelowH}.

\begin{prop}\label{prop:ExponentialForm} 
Assume \ref{itm:assumptions-h1} 
and \eqref{eq:superlinearOmega}. 
Let $\vartheta_\lambda$ be the solution to \eqref{eq:propStateConstraintBoundOmega}.
\begin{itemize}
\item[$\mathrm{(i)}$] Assume \eqref{eq:MonotoneWeakBelowH}.
For $x\in \R^n$, we have 
\begin{align}
	&\inf_{\gamma\in \mathcal{C}^-(x,\infty;\overline{\Omega})} \int_{-\infty}^0 e^{\lambda \alpha_{\gamma}^\lambda(s)} \Big(L(\gamma(s), \dot{\gamma}(s), 0) + c_\Omega(H)\Big)\;ds > -\infty,  \label{eq:BoundedBelowAC} \\
	\vartheta_\lambda(x) 
	&= \min_{\gamma(0) = x} \left\lbrace 
	\int_{-\infty}^0 e^{\lambda\alpha_{\gamma}^\lambda(s)} \Big(L(\gamma(s), \dot{\gamma}(s), 0) + c_\Omega(H)\Big)\;ds: \gamma\in \mathrm{Lip}\left((-\infty,0);\overline{\Omega}\right)  \right\rbrace. \label{eq:StateConstraintExpInfityLipschitz}
\end{align}

\item[$\mathrm{(ii)}$] 
Assume 
\begin{align}
    & u\mapsto H(x,p,u) - \underline{\kappa}\cdot u \;\text{is nondecreasing for all}\;(x,p) \label{eq:MonotoneStrongBelowH}.
\end{align}
Then  
\begin{equation}\label{eq:InfiniteExpOptimal}
	\vartheta_\lambda (x) = \inf_{\gamma\in \mathcal{C}^-(x,\infty;\overline{\Omega})} 
	\int_{-\infty}^0 e^{\lambda\alpha_{\gamma}^\lambda (s)} \Big(L(\gamma(s), \dot{\gamma}(s), 0) + c(H)\Big)\;ds.
\end{equation}	
\end{itemize}
\end{prop}

\begin{proof} 
For \( \gamma \in \mathcal{C}^-(x, \infty; \overline{\Omega}) \), unbounded \( \dot{\gamma} \) may prevent \( \alpha_{\gamma}^\lambda(s) \to -\infty \) as \( s \to -\infty \). \medskip 

{\bf Part (i). } By \eqref{eq:SuperLinearLtilde}, there exists \( M_\Omega \geq 1 \) such that 
\begin{equation*}
	\min_{x\in \overline{\Omega}}\frac{L(x,v,0)}{|v|} > 1+|c_\Omega(H)| \qquad\text{if}\; |v|\geq M_\Omega\geq 1.
\end{equation*}
Let us define $I_1 = \left\lbrace s<0: 
	|\dot{\gamma}(s)| > M_\Omega
	\right\rbrace$ and $I_2 = \left\lbrace s<0: 
	|\dot{\gamma}(s)| \leq M_\Omega
	\right\rbrace$. 
\begin{itemize}
\item For $I_1$, we observe that
\begin{align*}
	s\in I_1  \qquad\Longrightarrow\qquad  L(\gamma(s), \dot{\gamma}(s), 0) + c_\Omega(H) \geq |\dot{\gamma}(s)| + |c_\Omega(H)|\cdot|\dot{\gamma}(s)| + c_\Omega(H)\geq |\dot{\gamma}(s)|
\end{align*}
since $|\dot{\gamma}(s)|\geq M_\Omega \geq 1$. Therefore
\begin{align}\label{eq:I1}
	\int_{I_1} 
	e^{\lambda \alpha_{\gamma}^\lambda(s)} 
	\Big(L(\gamma(s), \dot{\gamma}(s), 0) + c_\Omega(H)\Big)\;ds  \geq 0. 
\end{align}
\item 
For $I_2$, by \eqref{eq:MonotoneWeakBelowH} (which is guaranteed by \ref{itm:assumptions-h3}), 
if $|\dot{\gamma}(s)|\leq M_\Omega$ then there exists $\underline{\kappa}>0$ such that $\kappa_{\lambda}(\gamma(s), \dot{\gamma}(s)) \leq -\underline{\kappa}$ if $|\dot{\gamma}(s)|\leq M_\Omega$. Thus
\begin{align*}
	\alpha_{\gamma}^\lambda(s) = \int_s^0 \kappa_{\lambda}(\gamma(s), \dot{\gamma}(s))\;ds \leq -\underline{\kappa}s \qquad\Longrightarrow\qquad 0\leq \int_{I_2} e^{\lambda \alpha_{\gamma}^\lambda(s)}\;ds  \leq \frac{1}{\lambda \underline{\kappa}}. 
\end{align*}
\end{itemize}	
Together with $L(x,v,0) + c(H) \geq -c_0$ for all $(x,v)\in \R^n\times\R^n$ from \eqref{eq:BoundL}, we have
\begin{align}\label{eq:I2}
	\int_{I_2} e^{\lambda \alpha_{\gamma}^\lambda(s)} \Big(L(\gamma(s), \dot{\gamma}(s), 0) + c_\Omega(H)\Big)\;ds \geq -c_0 \int_{I_2} e^{\lambda \alpha_{\gamma}^\lambda(s)}\;ds \geq -\frac{c_0}{\lambda \underline{\kappa}}. 
\end{align}
From \eqref{eq:I1} and \eqref{eq:I2} we obtain the boundedness from below \eqref{eq:BoundedBelowAC}. \medskip

To show \eqref{eq:StateConstraintExpInfityLipschitz}, we note that from \eqref{eq:estimateEXPAC} in the proof of Proposition \ref{prop:StateConstraintProperties}, we have
\begin{align}\label{eq:Lipa}
	\vartheta_\lambda(\gamma(0)) 
	\leq 
	e^{\lambda \alpha_{\gamma}^\lambda(-t)} \vartheta_\lambda(\gamma(-t)) 
	 + 
	 \int_{-t}^0 e^{\lambda\alpha_{\gamma}^\lambda(s)} 
	 \Big(L(\gamma(s), \dot{\gamma}(s), 0) + c_\Omega(H)\Big)\;ds 
	 \qquad\text{for any}\;t>0,
\end{align}
for any $\gamma \in \mathrm{Lip}((-\infty,0];\overline{\Omega})$ with $\gamma(0)=x$.
Under 
\eqref{eq:MonotoneWeakBelowH}, 
as $|\dot{\gamma}(s)|\leq C_\gamma$ for a.e. $s\in (-\infty,0)$, we have $\kappa_{\lambda}(\gamma(s), \dot{\gamma}(s)) \leq -\underline{\kappa}_\gamma$ for a.e. $s\leq 0 $. Let $t\to \infty$ in \eqref{eq:Lipa} we deduce that 
\begin{align*}
	\vartheta_\lambda(x) \leq \int_{-\infty}^0 e^{\lambda\alpha_{\gamma}^\lambda(s)} \Big(L(\gamma(s), \dot{\gamma}(s), 0) + c_\Omega(H)\Big)\;ds,
\end{align*}
thanks to $\alpha_{\gamma}^\lambda(s) \leq \underline{\kappa}_\gamma s$ for a.e. $s\leq 0$ and $\vartheta_\lambda$ is uniformly bounded. Hence, we have shown that 
\begin{align*}
	\vartheta_\lambda(x) \leq \inf_{\gamma(0)=x } \left\lbrace
	\int_{-\infty}^0 e^{\lambda\alpha_{\gamma}^\lambda(s)} \Big(L(\gamma(s), \dot{\gamma}(s), 0) + c_\Omega(H)\Big)\;ds : \gamma\in \mathrm{Lip}\left((-\infty,0);\overline{\Omega}\right)\right\rbrace. 
\end{align*}
By Proposition \ref{prop:ExistenceOfMinimizerTimeOmega}, we can find $\gamma\in \mathcal{C}^-(x,\infty;\overline{\Omega})$ such that $|\dot{\gamma}(s)|\leq C_\Omega$ and
\begin{align*}
	\vartheta_\lambda(\gamma(0)) - \vartheta_\lambda(\gamma(-t)) = \int_{-t}^0 \Big(L\big(\gamma(s),\dot{\gamma}(s), \lambda \vartheta_\lambda)\big)+c_\Omega(H)\Big)\;ds \qquad\text{for any}\;t>0.
\end{align*}
Similar to Proposition \ref{prop:StateConstraintProperties}, we have
\begin{align*}
	\vartheta_\lambda(\gamma(0)) = e^{\lambda \alpha_{\gamma,t}(-t)} \vartheta_\lambda(\gamma(-t))  + \int_{-t}^0 e^{\lambda\alpha_{\gamma}^\lambda(s)} \Big(L(\gamma(s), \dot{\gamma}(s), 0) + c_\Omega(H)\Big)\;ds \qquad\text{for any}\;t>0.
\end{align*}
Let $t\to \infty$, thanks to the fact that $\kappa_\lambda(\gamma(s), \dot{\gamma}(s)) \leq -\underline{\kappa}_\gamma$ for a.e. $s\leq 0$ we obtain the conclusion \eqref{eq:StateConstraintExpInfityLipschitz}. \medskip

{\bf Part (ii). }
If we assume \eqref{eq:MonotoneStrongBelowH}, then for any $\gamma\in \mathcal{C}^-(x,\infty;\overline{\Omega})$ we always have $\alpha^\lambda_\gamma(s) \leq \underline{\kappa} s$ for $s\leq 0$, thus we can repeat the argument above for \eqref{eq:StateConstraintExpInfityLipschitz} but with $\gamma\in \mathcal{C}^-(x,\infty;\overline{\Omega})$ instead, yielding \eqref{eq:InfiniteExpOptimal}. 
\end{proof}

\begin{rmk}\label{rmk:ExpSmall} From \eqref{eq:BoundedBelowAC}, we can define 
$\hat{\vartheta}_\lambda(x) = 
	\inf_{\gamma\in \mathcal{C}^-(x,\infty;\overline{\Omega})} 
	\int_{-\infty}^0 
	e^{\lambda \alpha_{\gamma}^\lambda (s)} 
	\Big(
		L(\gamma(s), \dot{\gamma}(s), 0) +c_\Omega(H)
	\Big)\;ds$. 
Then we can show that
\begin{align}\label{eq:formulavLambda}
	\hat{\vartheta}_\lambda(x) = \inf_{\gamma\in \mathcal{C}^-(x,t;\overline{\Omega})} 
	\left\lbrace 
	e^{\lambda \alpha_{\gamma}^\lambda(-t)} \hat{\vartheta}_\lambda(\gamma(-t))
	+
	\int_{-t}^0 e^{\lambda \alpha_{\gamma}^\lambda(s)}
	\Big(L(\gamma(s), \dot{\gamma}(s),0)+c_\Omega(H)\Big)\;ds 
	\right\rbrace. 
\end{align}
We have \( \hat{\vartheta}_\lambda \leq \vartheta_\lambda \), but under 
\eqref{eq:MonotoneWeakBelowH}
alone, it is not trivial that \( \hat{\vartheta}_\lambda = \vartheta_\lambda \). Furthermore, the uniform assumption \eqref{eq:MonotoneStrongBelowH} ensures that the state-constraint solution behaves as in the discounted case.
\end{rmk}
\color{black}

\subsection{Mather measures for state-constraint problem} 
Let \( (X, \mathcal{B}_X) \) be a measurable space with \( X \subset \mathbb{R}^n \) and \( \mathcal{B}_X \) the Borel \( \sigma \)-algebra. Denote by \( \mathcal{P}(X) \), \( \mathcal{R}(X) \), and \( \mathcal{R}^+(X) \) the sets of probability measures, Radon measures, and nonnegative Radon measures on \( X \), respectively. Throughout, we assume \( \Omega \subset \mathbb{R}^n \) is a bounded, open, connected domain with \( C^2 \) boundary. We recall the definition and properties of Mather measures on a bounded domain with state-constraint from \cite[Section 2]{tu_generalized_2024}.

\begin{defn}[Closed measures] \label{defn:mat-mea} Let $\Omega \subset \R^n$ be a bounded, open, connected subset with $C^2$ boundary. 
A probability measure $\mu\in \mathcal{P}(\overline{\Omega}\times\R^n)$ is called {\em holonomic} (or \emph{closed}) if it satisfies
\begin{itemize}
    \item[(i)] $\int_{\overline{\Omega}\times \R^n} |v|\;d\mu(x,v) < \infty$; 
    \item[(ii)] $\int_{\overline{\Omega}\times \R^n}  v\cdot D\phi(x)\;d\mu(x,v) = 0$ for every $\phi\in C^1(\overline{\Omega})$. 
\end{itemize}
We denote by $\mathcal{C}(\overline{\Omega}\times \R^n)$ the set of all holonomic measures in $\overline{\Omega}\times \R^n$.
\end{defn}

\begin{lem}\label{lem:lowercH} Let \( \mathcal{H} : \overline{\Omega} \times \mathbb{R}^n \to \mathbb{R} \) satisfy 
\ref{itm:assumptions-h1}
with the third argument fixed at \( u = 0 \). If \( u \in C(\overline{\Omega}) \) is a viscosity subsolution of \( \mathcal{H}(x, Du) = c \) in \( \Omega \), then the Legendre transform $\mathcal{L} = \mathcal{H}^*$ satisfies
\begin{equation}\label{eq:minMeasuresOmega}
	\int_{\overline \Om\times\R^n}\mathcal{L}(x,v)\; d\mu(x,v)\geq -c \qquad\text{for all}\; \mu\in \mathcal{C}(\overline{\Omega}\times \R^n). 
\end{equation}
As a consequence, under 
\ref{itm:assumptions-h1} and \eqref{eq:MonotoneWeakBelowH},
if $\vartheta_\lambda\in C(\overline{\Omega})$ solves \eqref{eq:propStateConstraintBoundOmega}, then for all $\mu\in \mathcal{C}(\overline{\Omega}\times \R^n)$ we have 
\begin{equation}\label{eq:minMeasures}
	\int_{\overline{\Omega}\times \R^n} L(x,v,0)\;d\mu(x,v) \geq -c_\Omega(H), \qquad \int_{\overline{\Omega}\times \R^n} L(x,v,\lambda \vartheta_\lambda(x))\;d\mu(x,v) \geq -c_\Omega(H).
\end{equation}
\end{lem}
\begin{proof} Let $u\in C(\Omega)$ be a viscosity subsolution to $\mathcal{H}(x,Du) \leq c_\Omega(\mathcal{H})$ in $\Omega$. By 
\eqref{eq:assumptions-LocalCoercive} in \ref{itm:assumptions-h1} 
we obtain that $u\in \mathrm{Lip}(\overline{\Omega})$, thus, similar to \cite[Lemma 2.10]{tu_generalized_2024}, for each $\varepsilon>0$, there exists a smooth function $u^\varepsilon\in C(\overline{\Omega})$ such that $\mathcal{H}(x,Du^\varepsilon(x), 0) \leq  c + \mathcal{O}(\varepsilon)$ in $\Omega$ in the classical sense. By Fenchel-Young's inequality we have
\begin{align*}
	v\cdot Du^\varepsilon(x) \leq \mathcal{L}(x,v) + \mathcal{H}(x,Du^\varepsilon(x)) \leq \mathcal{L}(x,v) + c + \mathcal{O}(\varepsilon),  \qquad  (x,v) \in \overline{\Omega} \times \R^n. 
\end{align*}
For \( \mu \in \mathcal{C}(\overline{\Omega} \times \mathbb{R}^n) \), integrating and letting \( \varepsilon \to 0^+ \) yields \eqref{eq:minMeasuresOmega}. The results in \eqref{eq:minMeasures} follow by applying \( \mathcal{H}(x,p) = H(x,p,0) \) and \( \mathcal{H}(x,p) = H(x,p,\lambda \vartheta_\lambda(x)) \), which satisfies the assumptions via Lemma \ref{lem:HTilde}.
\end{proof}

\begin{defn}[Mather measures]\label{defn:MatherMeasuresOmega} Assume 
\ref{itm:assumptions-h1}.
We define the set of Mather measures \( \mathfrak{M}(\Omega) \) as the set of closed measures \( \mu \in \mathcal{C}(\overline{\Omega} \times \mathbb{R}^n) \) such that
\begin{align*}
	\int_{\overline{\Omega}\times \R^n} L(x,v,0)\;d\mu(x,v) = -c_\Omega(H). 
\end{align*}
\end{defn}

Lemma \ref{lem:verification} and Theorem \ref{thm:VanishingOnStateConstraintBounded} are similar to results in \cite{tu_generalized_2024}. Note that the approach in \cite{tu_generalized_2024} used \( \kappa_\lambda = \partial_u L \), which led to an additional error term, whereas we obtain an exact formula here. Since the proofs require only minor modifications, we omit them.

\begin{lem}[{\cite[Lemma 2.8]{tu_generalized_2024}}]\label{lem:verification} 
Assume \ref{itm:assumptions-h1}, \eqref{eq:MonotoneWeakBelowH}, \eqref{eq:MonotoneWeakAboveH}, and \eqref{eq:superlinearOmega}. 
For each \(\lambda \in (0,1)\) and \(z \in \overline{\Omega}\), let \(\gamma_\lambda \in \mathcal{C}^-(z,\infty;\overline{\Omega})\) be a minimizing curve for \(\vartheta_\lambda\) (see \emph{Proposition \ref{prop:StateConstraintProperties}}). 
We denote by $\mu^\lambda_\gamma$ the Radon measure in $\mathcal{R}(\overline{\Omega}\times \R^n )\cap \mathcal{P}(\overline{\Omega}\times \R^n)$ such that 
\begin{align}\label{eq:DiscountedMeasure}
	\int_{\overline{\Omega}\times \R^n} \phi(x,v)\;d\mu^\lambda_\gamma(x,v) = \frac{\int_{-\infty}^0 e^{\lambda \alpha_{\gamma}^\lambda(s)} \phi(\gamma(s), \dot{\gamma}(s))\;ds}{\int_{-\infty}^0 e^{\lambda \alpha_{\gamma}^\lambda(s)} \;ds}, \qquad \phi \in C_c(\overline{\Omega}\times \R^n),
\end{align}
where $	\alpha_\gamma^\lambda(s) = \int_s^0 \kappa_\lambda(\gamma(s), \dot{\gamma}(s))\;ds$ for $s\leq 0$. The associated measure \(\mu^\lambda := \mu^\lambda_{\gamma_\lambda}\) is called a discounted measure for \(\vartheta_\lambda\) at \(z\). Along a subsequence \( \lambda_j \to 0^+ \), we have \( \mu^{\lambda_j} \rightharpoonup \mu \) in the weak$^*$ topology, with \( \mu \in \mathfrak{M}(\Omega) \). 
\end{lem}

\color{magenta}
\color{black}

\begin{thm}[{\cite[Theorem 1.1]{tu_generalized_2024}}]
\label{thm:VanishingOnStateConstraintBounded} 
Assume \ref{itm:assumptions-h1}, \ref{itm:assumptions-h3}, and \eqref{eq:superlinearOmega}. 
Then the solution \( \vartheta_\lambda \) to \eqref{eq:propStateConstraintBoundOmega} converges uniformly in \( C(\overline{\Omega}) \) to a function \( \vartheta_0 \in C(\overline{\Omega}) \) as \( \lambda \to 0^+ \), where \( \vartheta_0 \) is a viscosity solution to \eqref{eq:ErgodicStateOmega}, given by
$\vartheta_0 = \sup \mathcal{E}$, where
\begin{align*}
	\mathcal{E}(\Omega) = \left\lbrace w\in C(\Omega):H(x,Dw,0)\leq c_\Omega(H)\;\text{in}\;\Omega, \int_{\overline{\Omega}\times \R^n} \partial_uL(x,v,0)\cdot w(x)\;d\mu(x,v) \geq 0\;\text{for all}\;\mu \in \mathfrak{M}(\Omega) \right\rbrace. 
\end{align*}
\end{thm}

\section{The maximal solution via Perron Method} \label{sec:MaximalSolution}

Let \( \Omega \) be open, bounded, connected, and has \( C^2 \) boundary.
From the definitions of the ergodic constants $c_\Omega(H)$ and $c(H)$ in \eqref{eq:ergodicOmegaH} and \eqref{eq:c(H)}, respectively, we have
\begin{equation}\label{eq:increasingconstants}
    c_\Omega(H) \leq c_{\Omega'}(H) \leq c(H)  \qquad\text{for all}\;\Omega\subset \Omega' \subset \R^n. 
\end{equation}

\begin{lem}[{\cite[Proposition 5.1]{ishii_vanishing_2020}}]\label{lem:ErgodicConstantTheSame} Assume 
\ref{itm:assumptions-h1}, \ref{itm:assumptions-h2},
and let \(\mathcal{A}\) denote the Aubry set of $H$ in $\R^n$ from \emph{Definition \ref{defn:AubrySet}}. If $\Omega\subset \R^n$ is a domain such that \(\mathcal{A} \subset \Omega\), then \(c_\Omega(H) = c(H)\).
\end{lem}

\begin{cor}\label{coro:UniformBoundedLocally} 
Assume \ref{itm:assumptions-h1}--\ref{itm:assumptions-h3}. 
If \( u \in C(\mathbb{R}^n) \) is a subsolution to \( H(x, Du, \lambda u) \leq c(H) \) in \(\mathbb{R}^n\), then for any bounded domain \(\Omega\) with $C^2$ boundary 
containing the Aubry set \(\mathcal{A}\), there exists a constant \( C_\Omega \), depending only on \(\Omega\) and \(H\), such that $\sup_{\overline{\Omega}} u \leq C_\Omega$. 
\end{cor}
\begin{proof} We may take \( C_\Omega = \sup_{\overline{\Omega}} \vartheta_\lambda \), where \( \vartheta_\lambda \) is a state-constraint solution to \eqref{eq:propStateConstraintBoundOmega}. By the comparison principle (similar to the proof of Proposition~\ref{prop:UniformBoundedLocally}), any subsolution \( u \) of \( H(x, Du, \lambda u) \leq c(H) \) in \( \mathbb{R}^n \) satisfies \( u \leq \vartheta_\lambda \leq C_\Omega \) in $\Omega$.
\end{proof}

\begin{prop}[Perron method]\label{prop:PerronLipschitzNoContact} 
Assume \ref{itm:assumptions-h1}--\ref{itm:assumptions-h3}.
There exists a locally Lipschitz, bounded from below, and locally bounded from above viscosity solution to \eqref{eq:DP}, which is also the maximal subsolution in the class of continuous solutions. Furthermore, for every $R>0$, there exists a constant $C_R$ such that 
\begin{align}
	&u_\lambda(x) \geq -C_0 \qquad\text{for all}\;x\in \R^n  \label{eq:BoundedBelow}\\
	&|u_\lambda(x)| + |Du_\lambda(x)| \leq C_R \qquad\text{for a.e.}\;x\in \overline{B}_R(0). \label{eq:LipschitzMaximalSol}
\end{align}
\end{prop}

\begin{proof}[Proof of Proposition \ref{prop:PerronLipschitzNoContact}] 
By \ref{itm:assumptions-h1}, \ref{itm:assumptions-h2}, 
and the cutoff argument from Lemma \ref{lem:cutoffIshiiSiconofli}, there exists a subsolution $v$ to the ergodic problem $H(x,Dv,0) \leq c(H)$ in $\R^n$ that is Lipschitz and uniformly bounded in $\R^n$. Let 
\begin{equation}\label{eq:SubsolutionCutoff}
	w(x) = v(x)-\Vert v\Vert_{L^\infty(\R^n)}, \qquad x\in \R^n. 
\end{equation}
Thanks to the monotonicity in 
\ref{itm:assumptions-h1} 
$w$ is also a subsolution to $H(x,Dw,0) = c(H)$ in $\R^n$ and with $-C_0 \leq w \leq 0$ in $\R^n$, where $C_0 = 2\Vert v\Vert_{L^\infty(\R^n)}$. By 
monotonicity \ref{itm:assumptions-h3},
$w$ is also a subsolution to $H(x,Dw,\lambda w) = c(H)$ in $\R^n$ since $\lambda w \leq 0$, thus
    \begin{equation*}
        H(x,Dw, \lambda w) \leq H(x,Dw, 0) \leq c(H) \qquad\text{in}\;\R^n. 
    \end{equation*}
For $x\in \R^n$ and $\lambda > 0$, we define
\begin{align}\label{eq:PerronSol}
    u_\lambda(x):=\sup \left\lbrace v(x): w\leq v,v\;\text{is a subsolution}\;H(x,Dv,\lambda v) = c(H)\;\text{in}\;\R^n \right\rbrace.
\end{align}
The admissible set is nonempty since it contains \( w \). Moreover, by Corollary \ref{coro:UniformBoundedLocally}, any subsolution \( v \) to \( H(x, Dv, \lambda v) \leq c(H) \) in \( \mathbb{R}^n \) satisfies \( v(x) \leq C_R \) whenever \( \mathcal{A} \subset B_R(0) \). Therefore, \( u_\lambda \) is well-defined, uniformly bounded from below, locally bounded from above, and, furthermore, by 
\eqref{eq:assumptions-LocalCoercive} in \ref{itm:assumptions-h1}, 
\( u_\lambda \) is locally Lipschitz independent of \( \lambda \in (0,1) \). By virtue of Perron's method, \( u_\lambda \) is a viscosity solution to \eqref{eq:DP}.
\end{proof}

\begin{cor} \label{cor:max-sol-prop} 
Assume \ref{itm:assumptions-h1}--\ref{itm:assumptions-h3}. 
The family of maximal solutions \( \{u_\lambda\}_{\lambda \in (0,1)}\) to \eqref{eq:DP} is uniformly bounded from below, locally bounded from above, and equi-Lipschitz on any compact domains (with $C^2$ boundary) containing the Aubry set \( \mathcal{A} \). As a consequence, we have the following:
\begin{itemize}
    \item[$\mathrm{(i)}$] The sequence \(\lambda u_\lambda(\cdot) \to 0\) locally uniformly in \(\mathbb{R}^n\) as \(\lambda \to 0^+\).
    \item[$\mathrm{(ii)}$] 
    The set of locally uniform limits of \( u^{\lambda_j} \) as \( \lambda \to 0^+ \) is nonempty, and any such limit solves \eqref{eq:E}.
\end{itemize}
\end{cor}

Similar to \cite[Proposition 3.6]{ishii_vanishing_2020}, we can modify the Hamiltonian to be super-linear without loss of generality, as the primary focus is the solution \( u_\lambda \) to \eqref{eq:DP}.

\begin{prop}\label{prop:SuperlinearReduction} 
Assume \ref{itm:assumptions-h1}--\ref{itm:assumptions-h3}. 
Then there exists a Hamiltonian \( \widehat{H} \) satisfying the same assumptions and \( c(\widehat{H}) = c(H) \), such that for any compacts $K\subset \R^n$ and \( I \subset \R \),  
\begin{align}\label{eq:HSuperLinear}
	\lim_{|p|\to \infty} \min_{x\in K} \left\lbrace \frac{\widehat{H}(x,p,u)}{|p|}: u\in I  \right\rbrace= +\infty. 
\end{align}
Moreover, any subsolution of \( H(x,Du,0) \leq c(H) \) in \( \R^n \) is also a subsolution of \( \widehat{H}(x,Du,0) \leq c(H) \), and the solution \( u_\lambda \) from Proposition \ref{prop:PerronLipschitzNoContact} is also the maximal solution to \( \widehat{H}(x,Du,\lambda u) = c(H) \) in \( \R^n \).
\end{prop}

\begin{proof} We define  
\[
\widehat{H}(x,p,u) := H(x,p,u) + \left( \max\{0, H(x,p,u) - c(H) \} \right)^2.
\]  
As in \cite{ishii_vanishing_2020}, if \( H \) satisfies 
\ref{itm:assumptions-h1}, \ref{itm:assumptions-h2}, 
then so does \( \widehat{H} \), since these conditions depend only on \( H(x,p,0) \), and \( c(\widehat{H}) = c(H) \). Since \( \widehat{H} \geq H \), any subsolution of \( \widehat{H} = c(H) \) is also a subsolution of \( H = c(H) \). Thus, the maximal solution \( \widehat{u}_\lambda \) satisfies \( \widehat{u}_\lambda \leq u_\lambda \). Conversely, \( u_\lambda \) is also a subsolution of \( \widehat{H} = c(H) \), so \( \widehat{u}_\lambda \geq u_\lambda \). Hence, \( u_\lambda = \widehat{u}_\lambda \) and is maximal for both equations.
\end{proof}

\section{The variational formula for the maximal solution} \label{sec:variational}
This section aims to derive a variational formula for the maximal solution \( u_\lambda \) of \eqref{eq:DP}. We consider the associated evolution problem 
\begin{equation}\label{eq:Cauchy}
	\begin{cases}
	\begin{aligned}
		u_t + H(x,Du,\lambda u_\lambda(x)) &= 0 &&\text{in}\;\R^n\times (0,\infty), \\
		u(x,0) &= u_\lambda(x)&&\text{on}\;\R^n. 
	\end{aligned}
	\end{cases}
\end{equation}
To derive a variational formula for \( u_\lambda \), it suffices to show \( \tilde{u}(x,t) = u_\lambda(x) - c(H)t \) coincides with the value function \( u \) from optimal control. This equality may fail without further assumptions due to the lack of comparison in unbounded domains, and the fact that $u_\lambda$ may not be bounded.

\subsection{The Lax-Oleinik semigroup}

\begin{lem}\label{lem:HTilde} 
Assume \ref{itm:assumptions-h1}--\ref{itm:assumptions-h3}. 
Let us define
\begin{equation}\label{eq:Htiltde}
	\widetilde{H}(x,p) = H(x,p, \lambda u_\lambda(x)), \qquad (x,p)\in \R^n\times\R^n.
\end{equation}
Then \(\widetilde{H}\) is continuous, convex in \(p\) for each \(x \in \mathbb{R}^n\), and satisfies \(\min_{x \in K} \widetilde{H}(x,p) \to \infty\) as \(|p| \to \infty\) for any compact \(K \subset \mathbb{R}^n\). Its Legendre transform \(\widetilde{L}\) is given by
\begin{equation}\label{eq:Ltiltde}
	\widetilde{L}(x,v) = L(x,v, \lambda u_\lambda(x)), \qquad (x,v)\in \R^n\times\R^n.
\end{equation}
\end{lem}
We omit the proof of this Lemma as it is standard.

\begin{prop}\label{prop:valuefunction} 
Assume \ref{itm:assumptions-h1}, \ref{itm:assumptions-h2}. 
For \( (x,t)\in \R^n\times (0,\infty) \), we define the value function as 
\begin{equation}\label{eq:ValueFunction}
    u(x,t) = \inf_{\gamma\in \mathcal{C}^+(x,t;\R^n)}\left\lbrace \int_{0}^t L\big(\gamma(s), \dot{\gamma}(s), \lambda u_\lambda(\gamma(s))\big) \;ds +  u_\lambda(\gamma(0)) \right\rbrace .
\end{equation}
Then $u\in C(\R^n\times [0,\infty))\cap W^{1,\infty}_{\mathrm{loc}}(\R^n\times [0,\infty))$, and is a continuous viscosity solution to 
\begin{equation}\label{eq:CauchyContact}
	\begin{cases}
	\begin{aligned}
		u_t(x,t) + H(x,Du(x,t), \lambda u_\lambda(x)) &= 0 &&\qquad\text{in}\;\R^n\times (0,\infty), \\
		u(x,0) &= u_\lambda(x) &&\qquad\text{on}\;\R^n. 
	\end{aligned}
	\end{cases}
\end{equation}
Furthermore, $u(x,t) \geq \tilde{u}(x,t) = u_\lambda(x) - c(H)t$ for all $(x,t)\in \R^n\times (0,\infty)$. 
\end{prop}
We refer to \cite{mitake_asymptotic_2008} for the approach to this proposition, with suitable modifications to accommodate the unbounded domain.

\begin{prop}\label{prop:Minimizer} 
Assume \ref{itm:assumptions-h1}, \ref{itm:assumptions-h2}. 
For $(x,t)\in \R^n\times (0,\infty$), the value function $u(x,t)$ defined as in \eqref{eq:ValueFunction} admits a minimizer $\gamma\in\mathcal{C}^+(x,t;\R^n)$ such that
\begin{equation*}
    u(x,t) = u_\lambda(\gamma(0)) + \int_{0}^t L\big(\gamma(s), \dot{\gamma}(s), \lambda u_\lambda(\gamma(s))\big) \;ds . 
\end{equation*}
\end{prop}

\begin{proof} Due to Proposition \ref{prop:valuefunction}, the infimum is finite, since $u_\lambda$ is bounded from below by Proposition \ref{prop:PerronLipschitzNoContact}, and 
\begin{align*}
	\inf_{\gamma\in \mathcal{C}^+(x,t;\R^n)} \left\lbrace \int_{0}^t L\big(\gamma(s), \dot{\gamma}(s), \lambda u_\lambda(\gamma(s))\big) \;ds +  u_\lambda(\gamma(0))\right\rbrace = u(x,t) \geq u_\lambda(x) - c(H)t \geq -C_0 - c(H)t. 
\end{align*}
For simplicity, let us define 
\begin{align*}
	\Lambda(\gamma) &:= \int_{0}^t L\big(\gamma(s), \dot{\gamma}(s), \lambda u_\lambda(\gamma(s))\big) \;ds,                  &&\qquad \gamma \in \mathcal{C}^+(x,t;\R^n) \\
	J(\gamma) &:= \int_{0}^t L\big(\gamma(s), \dot{\gamma}(s), \lambda u_\lambda(\gamma(s))\big) \;ds + u_\lambda(\gamma(0)), &&\qquad \gamma \in \mathcal{C}^+(x,t;\R^n).
\end{align*}
Let $\gamma_k\in \mathcal{C}^+(x,t;\R^n)$ be a minimizing sequence such that $\lim_{k\to \infty} J(\gamma_k) =  u(x,t)$. Without loss of generality, assume (for large \( k \)) that \( \sup_k J(\gamma_k) \leq u(x,t) + 1 \).
Since $u_\lambda \geq -C_0$, we have
\begin{align*}
	\sup_{k\in \N} \Lambda(\gamma_k) = \sup_{k\in \N}\int_0^t L(\gamma(s), \dot{\gamma}(s), \lambda u_\lambda(\gamma(s)))\;ds  \leq u(x,t) +1 + C_0.
\end{align*}
Applying Lemma \ref{lem:minimizerUSC} to $\widetilde{H}$ and $\widetilde{L}$, we obtain a subsequence, still denoted by $\gamma_k$, and a curve $\gamma \in \mathcal{C}^-(x,t;\R^n)$ such that $\gamma_k \to \gamma$ uniformly in \( L^\infty([0,T]; \R^n) \) as \( k \to \infty \). Moreover,  
\[
\Lambda(\gamma) \leq \liminf_{k \to \infty} \Lambda(\gamma_k)
\quad \Longrightarrow \quad
J(\gamma) \leq \liminf_{k \to \infty} J(\gamma_k),
\]
since $\gamma_k(0) \to \gamma(0)$ by uniform convergence. Hence, \( u(x,t) = J(\gamma) \). 
\end{proof}

The result under which the value function u defined by \eqref{eq:ValueFunction} equals $\tilde{u}(x,t) = u_\lambda(x) - c(H)t$ relies on certain conditions. 
The boundedness of $u_\lambda$ from above, one of the required condition, is established in the following Lemma \ref{lem:PropertiesUStrongCoerNew}. 
\begin{prop}\label{prop:VariationalFomular} 
Assume \ref{itm:assumptions-h1}, \ref{itm:assumptions-h2}. 
Let \( u \) be given by \eqref{eq:ValueFunction}, and let \( u_\lambda \) denote the maximal solution to \eqref{eq:DP}. 
The equality \( u(x,t) = u_\lambda(x) - c(H)t \) holds under any of the following conditions:
\begin{enumerate}
    \item[$\mathrm{(i)}$] $u_\lambda$ is bounded from above and \eqref{eq:MonotoneWeakAboveH}; or $u_\lambda$ is bounded from above and \eqref{eq:assumptions-UniformBoundWithP};
    \item[$\mathrm{(ii)}$] \ref{itm:assumptions-pp2} and \eqref{eq:MonotoneWeakAboveH}; or \ref{itm:assumptions-pp2} and \eqref{eq:assumptions-UniformBoundWithP};
    \item[$\mathrm{(iii)}$] \ref{itm:assumptions-pp1}. 
\end{enumerate}

Consequently, if either condition holds, the maximal solution $u_\lambda$ to \eqref{eq:DP} admits the following representation formula:
\begin{equation}\label{eq:MinimizerTimePositive}
	u_\lambda(x) 
	= 
	\inf_{\gamma \in \mathcal{C}^+(x,t;\R^n)} 
	\left\lbrace 
	u_\lambda(\gamma(0)) 
	+
	\int_{0}^t 
	\left( 
		L\big(\gamma(s), \dot{\gamma}(s), \lambda u_\lambda(\gamma(s))\big) + c(H) 
	\right)\;ds 
	\right\rbrace 
	\quad \forall (x,t) \in \R^n\times \R^+.  
\end{equation}
Furthermore, for any \( R > 0 \), there exists \( C_R > 0 \) such that, if \( \gamma \in \mathcal{C}^+(x,t;\R^n) \) is a minimizer to \eqref{eq:MinimizerTimePositive}, then
\begin{align*}
    |\gamma(s)|\leq R \qquad\Longrightarrow\qquad |\dot{\gamma}(s)| \leq C_R \qquad\text{for a.e.}\;s\in (0,t).
\end{align*}
\end{prop}
\begin{proof} We have already obtain $\tilde{u} \leq u$ from Proposition \ref{prop:valuefunction}. \medskip

{\bf Case (i).} If $\sup_{x\in \R^n} u_\lambda(x) < \infty$, the Hamiltonian $\widetilde{H}(x,p) = H(x,p,\lambda u_\lambda(x))$ satisfies that, the equation
\begin{equation*}
	\begin{cases}
	\begin{aligned}
		w_t + \widetilde{H}(x,Dw) &= 0 &&\qquad \text{in}\;\R^n\times (0,\infty), \\
		w(x,0) &= u_\lambda(x) &&\qquad \text{on}\;\R^n.
	\end{aligned}
	\end{cases} 
\end{equation*} 
has a standard comparison principle (see \cite[Theorem 1.23]{tran_hamilton-jacobi_2021}, where we can modify $\widetilde{H}(x,p)$ when $p$ large so that a condition in the available comparison results in the literature). 
Indeed, by Lemma \ref{lem:cutoffIshiiSiconofli},
if \( \tilde{v} \in C(\R^n) \) is a viscosity subsolution of \( H(x, D\tilde{v}, 0) \leq c(H) \) in \( \R^n \), with \( |D\tilde{v}| \leq C \) and \( \tilde{v}(x) \to -\infty \) as \( |x| \to \infty \).
\begin{itemize}
    \item If \eqref{eq:MonotoneWeakAboveH} holds, then 
    \begin{equation}\label{eq:UsingA5Up}
    \begin{aligned}
    	\widetilde{H}(x, D\tilde{v}(x)) 
    	& \leq H(x,D\tilde{v}(x), 0) + \overline{\kappa}_C |\lambda u_\lambda(x)| \leq c(H) + \overline{\kappa}_C \cdot \lambda \cdot \sup_{x\in \R^n}  u_\lambda(x), \qquad x\in \R^n.
    \end{aligned}
    \end{equation}
    \item If \eqref{eq:assumptions-UniformBoundWithP} holds, then $\widetilde{H}(x, D\tilde{v}(x)) \leq C$ in $\R^n$ since $\sup_{x \in \R^n} \lambda u_\lambda(x) < \infty$.
\end{itemize}
For each fixed $\lambda \in (0,1)$ and $t \in (0,T]$ we have
\begin{equation*}
	\lim_{r\to \infty}  \inf \left\{u_\lambda(x) - c(H)t - \tilde{v}(x):|x|\geq r, t\in  [0,T]\right\} = +\infty. 
\end{equation*}
From Proposition \ref{prop:valuefunction} we have $u_\lambda(x) = u(x,0)$, by Theorem \ref{thm:ComparisonIshiiMod} we obtain that two solutions $u_\lambda(x) - c(H)t$ and $u(x,t)$, constructed as in \eqref{eq:ValueFunction} are the same, thus the conclusion follows. 

\medskip

{\bf Case (ii).} There are two scenarios:
\begin{itemize}
	\item If $\displaystyle \lim_{r\to \infty}\inf \{u_\lambda(x): |x|\geq r\} < \infty$, then $\displaystyle \sup_{x\in \R^n} |u_\lambda(x)| < \infty$, thus by the previous case we have $u\equiv \tilde{u}$.
	\item If $\displaystyle \lim_{r\to \infty}\inf \{u_\lambda(x): |x|\geq r\} = \infty$, then for $\theta\in (0,1)$ and any $T>0$, we have
	\begin{align*}
		&\lim_{r\to \infty}\inf \{ \tilde{u}(x,t) - \theta u_\lambda(x):  |x|\geq r, t\in (0,T)\} 
		= (1-\theta) \cdot \lim_{r\to \infty}\inf_{|x|\geq r}  u_\lambda(x) - c(H)T = +\infty
	\end{align*}
\end{itemize}
Let $\phi(x) = \theta u_\lambda(x)$ for $x\in \R^n$. By 
\ref{itm:assumptions-pp1} we have
\begin{align*}
	\widetilde{H}(x,D\phi(x)) 
		= H(x,\theta Du_\lambda(x), \lambda u_\lambda(x))   
		\leq H(x,Du_\lambda(x), \lambda u_\lambda(x)) +C_\theta \leq c(H)+C_\theta \qquad \text{in}\;\R^n.
\end{align*}
By The Comparison Principle in Theorem \ref{thm:ComparisonIshiiMod} applying to the supersolution $\tilde{u}$ and the subsolution $u$ to \eqref{eq:Cauchy}, we obtain $\tilde{u} \geq u$ and thus $\tilde{u} \equiv u$. 
\medskip

{\bf Case (iii).} From Lemma \ref{lem:cutoffIshiiSiconofli}, let $\tilde{v}\in C(\R^n)$ be the viscosity subsolution to $H(x,D\tilde{v},0)\leq c(H)$ in $\R^n$ such that $|D\tilde{v}|\leq C$, and $\tilde{v}(x)\to -\infty$ as $|x|\to \infty$. 
By the joint convexity of $H(x,p,u)$ in $(p,u)$, for any $\theta \in (0,1)$, we have that
\begin{align*}
	H(x, Dv_\theta(x), \theta  \cdot \lambda u_\lambda(x)) \leq c(H) \qquad\text{in}\;\R^n,
\end{align*}
where $v_\theta(x) = \theta u_\lambda(x) + (1-\theta) \tilde{v}(x)$ for $x\in \R^n$. By 
\ref{itm:assumptions-pp2} 
there exists a constant $C_\theta = C_\theta(\theta, c(H))$ such that
\begin{equation*}
	\widetilde{H}(x,Dv_\theta(x)) = H(x, Dv_\theta(x), \lambda u_\lambda(x)) \leq C_\theta \qquad\text{in}\;\R^n.
\end{equation*}
We can apply the Comparison Principle in Theorem \ref{thm:ComparisonIshiiMod} to obtain the conclusion, with $\phi = v_\theta$. \medskip

If \( \gamma \) is a minimizer, then for any \( t > 0 \), the map \( s \mapsto u_\lambda(\gamma(s)) \) is absolutely continuous on \( (-t, 0) \), and there exists \( p \in L^\infty((-t, 0); \R^n) \) such that  
\[
\frac{d}{ds} u_\lambda(\gamma(s)) = p(s) \cdot \dot{\gamma}(s), \quad p(s) \in \partial_c u_\lambda(\gamma(s)) \quad \text{a.e. } s \in (-t, 0),
\]
and
\[
|p(s)| \cdot |\dot{\gamma}(s)| \geq L(\gamma(s), \dot{\gamma}(s), \lambda u_\lambda(\gamma(s))) + c(H).
\]
If \( |\gamma(s)| \leq R \), then by Proposition \ref{prop:PerronLipschitzNoContact}, \( |p(s)| \leq \widehat{C}_R \) and \( u_\lambda(\gamma(s)) \leq C_R \), so the result follows from 
\eqref{eq:assumptions-LocalCoercive} in \ref{itm:assumptions-h1}.
\end{proof}

Next, we verify that $u_\lambda$ is bounded above and thus satisfies \eqref{eq:MinimizerTimePositive}, under condition (i) of Proposition \ref{prop:VariationalFomular}.

\begin{lem}\label{lem:PropertiesUStrongCoerNew} 
Assume \ref{itm:assumptions-h1}--\ref{itm:assumptions-h3} and \ref{itm:assumptions-pp3}. Then $\Vert Du_\lambda \Vert_{L^\infty(\R^n)} + \Vert u_\lambda \Vert_{L^\infty(\R^n)} \leq \widehat{C}_0$, and there is a constant $C$ such that, any minimizer $\gamma\in \mathcal{C}^+(x,t;\R^n)$ to \eqref{eq:MinimizerTimePositive} satisfies 
\begin{equation}\label{eq:boundedVelocity}
	|\dot{\gamma}(s)| \leq C \qquad\text{for a.e.}\;s \in [0,t].
\end{equation}
\end{lem}

\begin{proof} 
Let $C_0$ be the constant from Proposition \ref{prop:PerronLipschitzNoContact} satisfying $\lambda u_\lambda(x) \geq -C_0$. The monotonicity assumption \ref{itm:assumptions-h1} implies that $u_\lambda$ is a subsolution to $H(x,Du_\lambda, -C_0) = c(H) $ in $\R^n$.
Consequently, if the uniform coercivity condition \eqref{eq:assumptions-UniformCoercive} in \ref{itm:assumptions-pp1} holds, then $\Vert Du_\lambda \Vert_{L^\infty(\R^n)} \leq \widehat{C}_0$. \medskip

Additionally, if \eqref{eq:MonotoneStrongBelowH} 
and \eqref{eq:eq:assumptions-UniformBoundBelow} in \ref{itm:assumptions-pp3} hold then $\lambda \Vert u_\lambda \Vert_{L^\infty(\R^n)} \leq \widehat{C}_0$. Indeed, from \( |Du_\lambda(x)| \leq \widehat{C}_0 \) that
\begin{align*}
	c(H) -H(x, Du_\lambda(x),-C_0) 
	=  H(x, Du_\lambda(x),\lambda u_\lambda(x)) - H(x, Du_\lambda(x),-C_0)  
	\geq \underline{\kappa} (\lambda u_\lambda(x) + C_0)
\end{align*}
for a.e. \( x\in \mathbb{R}^n \). By \eqref{eq:eq:assumptions-UniformBoundBelow} in \ref{itm:assumptions-pp3} we can choose $\widehat{C}_0$ such that $H(x,Du_\lambda(x), -C_0) \geq 1-\widehat{C}_0$, then
\begin{align*}
	\underline{\kappa} \big(\lambda u_\lambda(x) + C_0 \big) \leq c(H) + \widehat{C}_0,
\end{align*}
which implies that $u_\lambda$ is bounded from above. \medskip

Lastly, under \eqref{eq:assumptions-UniformBoundWithP} in \ref{itm:assumptions-pp3} we have
\begin{equation}\label{eq:globalSuperLinearLInU}
	\sup_{x\in \R^n} \{ H(x,p,u): |p|\leq R, |u|\leq R\}  < \infty
	 \qquad\Longrightarrow\qquad 
	 \lim_{|v|\to \infty} \min_{x\in \R^n} \left\lbrace \frac{L(x,v,u)}{|v|}: |u|\leq R \right\rbrace = \infty. 
\end{equation}
Using $\lambda u_\lambda \leq \widehat{C}_0$ and the principle of optimality, for $0<a<b<t$ we have 
\begin{align*}
	u_\lambda (\gamma(b))-u_\lambda(\gamma(a)) &= 
	\int_a^b\Big(L\big(\gamma(s),\dot{\gamma}(s),\lambda u_\lambda (\gamma(s))\big)+c(H)\Big)\; ds 
	\geq \int_a^b\Big(L\big(\gamma(s),\dot{\gamma}(s),\widehat{C}_0)\big)+c(H)\Big)\; ds.
 \end{align*}
Using $|Du_\lambda|\leq \widehat{C}_0$, we obtain
 \begin{align*}
 	\frac{1}{b-a} \int_a^b\Big(L\big(\gamma(s),\dot{\gamma}(s),\tilde{C}_0)\big)+c(H)\Big)\; ds \leq \widehat{C}_0\frac{\left|\gamma(b)-\gamma(a)\right|}{b-a}, \qquad 0 < a<b<t. 
 \end{align*}
If $\gamma$ is differentialble at $t_0 \in (0,t)$, we obtain
\begin{align*}
	L(\gamma(t_0), \dot{\gamma}({t_0}), \widehat{C}_0) + c(H) \leq \widehat{C}_0\cdot |\dot{\gamma}(t_0)|. 
\end{align*}
Using \eqref{eq:globalSuperLinearLInU}, there must be a constant $C$ indpendent of $x,t$ such that $|\dot{\gamma}(t_0)|  \leq C$.  
\end{proof}

\begin{cor}\label{cor:Minimizer} Assume \ref{itm:assumptions-h1}--\ref{itm:assumptions-h3} and \ref{itm:assumptions-pp3}. Then the solution $u_\lambda$ to \eqref{eq:DP} satisfies \eqref{eq:MinimizerTimePositive}.
\end{cor}

\begin{proof} By Lemma \ref{lem:PropertiesUStrongCoerNew}, \( \sup_{x\in \mathbb{R}^n} \lambda u_\lambda(x) \leq \widehat{C}_0 \). Then, by Proposition \ref{prop:VariationalFomular} (i), we obtain the conclusion.
\end{proof}

\subsection{The Exponential form of the variational formula for the maximal solution} In this subsection, we transform \eqref{eq:MinimizerTimePositive} into an exponential form.

\begin{defn}[Discounted index for the global solution] 
Assume \ref{itm:assumptions-h1}, \ref{itm:assumptions-h2}. 
For $(x,v)\in \R^n\times\R^n$, we define
\begin{equation}\label{eq:DiscountedIndex}
    K_\lambda(x,v) = 
    \begin{cases}
    \partial_u^+ L(x,v,0) &\text{if}\; u_\lambda (x) = 0, \vspace{0.1cm}\\ 
    \dfrac{L(x,v,\lambda u_\lambda(x)) - L(x,v,0)}{\lambda u_\lambda(x)} &\text{if}\; u_\lambda (x) \neq 0.
    \end{cases}
\end{equation}
\end{defn}
\begin{rmk} We use $K_\lambda$ for the index of the global solution $u_\lambda$ on $\mathbb{R}^n$, distinguishing it from $\kappa_\lambda$, the discounted index for the state-constraint solution $\vartheta_\lambda$ on a bounded domain \emph{(Definition \ref{defn:DiscountedIndexOmega})}. Since \( u \mapsto L(x,v,u) \) is non-increasing, we have \( K_\lambda(x,v) \leq 0 \) for all \( (x,v) \in \mathbb{R}^n \times \mathbb{R}^n \).
\end{rmk}

\begin{defn}
Assume \ref{itm:assumptions-h1}, \ref{itm:assumptions-h2}. 
For \( (x,t) \in \mathbb{R}^n \times (0,\infty) \), we define:
\begin{itemize}
\item[$\mathrm{(i)}$] \(\mathcal{M}_\lambda(x,t)\): the set of curves \(\gamma \in \mathcal{C}^-(x,t;\mathbb{R}^n)\) satisfying \eqref{eq:MinimizerTimePositive};
\item[$\mathrm{(ii)}$] \(\mathcal{M}_\lambda(x,\infty)\): the set of curves \(\gamma \in \mathcal{C}^-(x,\infty;\mathbb{R}^n)\) such that \(\gamma \in \mathcal{M}_\lambda(x,t)\) for all \(t>0\).
\end{itemize}

\end{defn}

\begin{cor}\label{cor:nonemtpyM} 
Assume \ref{itm:assumptions-h1}, \ref{itm:assumptions-h2} 
and that \( u_\lambda \) satisfies \eqref{eq:MinimizerTimePositive}. Then, for every \( (x,t) \in \mathbb{R}^n \times (0,\infty) \), the sets \( \mathcal{M}_\lambda(x,t) \) and \( \mathcal{M}_\lambda(x,\infty) \) are nonempty.
\end{cor} 

\begin{proof} If \( u_\lambda \) satisfies \eqref{eq:MinimizerTimePositive}, then \( \mathcal{M}_\lambda(x,t) \neq \emptyset \) by Proposition \ref{prop:Minimizer}. A sequence of minimizers over finite horizons can be linked, as in \cite[Corollary 6.2]{ishii_asymptotic_2008} or the proof of Proposition \ref{prop:ExistenceOfMinimizerTimeOmega}, to construct \( \gamma \in \mathcal{M}_\lambda(x,\infty) \).
\end{proof}

\begin{prop}\label{prop:exponentialNegative} 
Assume \ref{itm:assumptions-h1}--\ref{itm:assumptions-h2}
and that $u_\lambda$ satisfies \eqref{eq:MinimizerTimePositive}. 
We have the followings.
\begin{itemize}
\item[$\mathrm{(i)}$] For all $(x,t)\in \R^n\times [0,\infty)$ we have
\begin{align}\label{eq:dyn-pro-bisNegative}
	u_\lambda(x) &= \inf_{\gamma\in\mathcal{C}^-(x,t;\R^n)} 
	\left\lbrace  
	e^{ \lambda \beta_{\gamma}^\lambda(-t)} u_\lambda(\gamma(-t)) 
	+ 
	\int_{-t}^0 e^{ \lambda \beta_{\gamma}^\lambda(s)} 
	\Big( 
		L\big(\gamma(s), \dot{\gamma}(s), 0\big) + c(H)  
	\Big)\;ds 
	\right\rbrace,
\end{align}
where 
\begin{equation}\label{eq:beta}
	\beta_{\gamma}^\lambda(s) = \int_{s}^0  K_\lambda(\gamma(\tau), \dot{\gamma}(\tau))\;d\tau, \qquad s\in [-t,0].
\end{equation}

\item[$\mathrm{(ii)}$] Let \( \gamma \in \mathcal{C}^-(x,t;\R^n) \) minimize \eqref{eq:dyn-pro-bisNegative}. Then, for \( -t < b < a < 0 \), we have:
\begin{align}\label{eq:optimalityNegativeEXP}
	  e^{\lambda \beta_\gamma^\lambda (a)}u_\lambda(\gamma(a)) 
	  =  
	  e^{\lambda \beta_\gamma^\lambda (b)}u_\lambda(\gamma(b)) 
	  + 
	  \int_{b}^{a} 
	  e^{ \lambda \beta_\gamma^\lambda(s)}
	  \big( L\big(\gamma(s), \dot{\gamma}(s), 0\big) + c(H)  \big)\;ds. 
\end{align}

\item[$\mathrm{(iii)}$] Any minimizer \( \gamma \in \mathcal{C}^-(x,t;\mathbb{R}^n) \) of \eqref{eq:MinimizerTimePositive} also minimizes \eqref{eq:dyn-pro-bisNegative} and vice versa, so the set of minimizers to \eqref{eq:dyn-pro-bisNegative} identical to $\mathcal{M}_\lambda(x,t)$, and thus is nonempty.

\item[$\mathrm{(iv)}$] There exists $\gamma\in \mathcal{C}^-(x,\infty;\R^n)$ such that
\begin{equation}\label{eq:InfinityMinimizerExp}
	u_\lambda(x) 
	=
	e^{ \lambda \beta_{\gamma}^\lambda(-t)} u_\lambda(\gamma(-t)) 
	+ 
	\int_{-t}^0 
	e^{ \lambda \beta_\gamma^\lambda(s)} 
		\Big( L\big(\gamma(s), \dot{\gamma}(s), 0\big) + c(H)  \Big)\;ds 
	\qquad\text{for all}\;t>0. 
\end{equation}
\end{itemize}
\end{prop}

\begin{proof} The proof follows with minor modifications from the state-constraint case in Proposition \ref{prop:StateConstraintProperties}.  \medskip

{\bf Part (i).} Take $\gamma\in \mathcal{C}^-(x,t;\R^n)$, by Proposition \ref{prop:curve} we have
\begin{equation}\label{eq:timeMin}
	u_\lambda(\gamma(a)) \leq u_\lambda(\gamma(b)) +  \int_{b}^a\Big( L\left((\gamma(s), \dot{\gamma}(s), \lambda u_\lambda(\gamma(s))\right) + c(H)\Big)\;ds.
\end{equation}
By \cite[Proposition 4.1]{mitake_asymptotic_2008}, 
$s\mapsto u_\lambda(\gamma(s))$ is absolutely continuous on $[-t,0]$, thus 
\begin{align*}
	\frac{d}{ds}\big(u_\lambda(\gamma(s))\big) 
	\leq 
	L
	\big(
	\gamma(s), \dot{\gamma}(s), 
	\lambda u_\lambda (\gamma(s))
	\big)
	 + c(H) 
	\qquad\text{for a.e.}\;s\in (-t,0). 
\end{align*}
For a.e. $s\in (-t,0)$ we have
\begin{align*}
	\frac{d}{ds}\left(e^{\lambda \beta_{\gamma}^\lambda(s)} 
	u_\lambda (\gamma(s))   \right) 
	&= e^{\lambda \beta_{\gamma}^\lambda(s)}
	\left(
	-\lambda 
	u_\lambda(\gamma(s)) 
	\cdot K_{\lambda}(\gamma(s), \dot{\gamma}(s))  
	+ 
	\frac{d}{ds}\big(u_\lambda(\gamma(s))\big) 
	\right) \\
	&\leq 
	e^{\lambda \beta_{\gamma}^\lambda (s)}
	\Big(-\lambda u_\lambda (\gamma(s)) 
	\cdot K_{\lambda }(\gamma(s), \dot{\gamma}(s))  
	+ 
	L(\gamma(s), \dot{\gamma}(s), \lambda u_\lambda(\gamma(s))) + c(H)\Big) \\
	&=
	e^{\lambda \beta_{\gamma}^\lambda(s)}
	\Big( L(\gamma(s), \dot{\gamma}(s), 0) +c(H)\Big).
\end{align*}
Integrating over $[-t,0]$, we obtain
\begin{align*}
	u_\lambda(x) = 
	e^{\lambda \beta_{\gamma}^\lambda(0)}
	u_\lambda(\gamma(0)) 
	\leq 
	e^{\lambda \beta_{\gamma}^\lambda(-t)}
	u_\lambda(\gamma(-t))  
	+ 
	\int_{-t}^0 e^{\lambda\beta_{\gamma}^\lambda(s)} 
	\Big(L(\gamma(s), \dot{\gamma}(s), 0) + c(H)\Big)\;ds. 
\end{align*}
Since $\gamma \in \mathcal{C}^-(x,t;\R^n)$ is chosen arbitrarily, we obtain \eqref{eq:dyn-pro-bisNegative}.
\medskip

{\bf Part (ii).} From \eqref{eq:dyn-pro-bisNegative} in Proposition \ref{prop:exponentialNegative}, for any $\gamma\in \mathcal{C}^-(x,t;\R^n)$ we have
\begin{align*}
	e^{\lambda \beta_\gamma^\lambda(a)}u_\lambda(\gamma(a)) 
	  \leq  
	  e^{\lambda \beta_\gamma^\lambda (b)}u_\lambda(\gamma(b)) 
	  + 
	  \int_{b}^{a} 
	  e^{ \lambda \beta_\gamma^\lambda(s)}
	  \big( L\big(\gamma(s), \dot{\gamma}(s), 0\big) + c(H)  \big)\;ds.
\end{align*}
On the other hand, by \eqref{eq:dyn-pro-bisNegative} we have
\begin{align*}
	  e^{\lambda\beta_\gamma^\lambda (-t)}u_\lambda(\gamma(-t)) 
	  + 
	  \int_{-t}^{b} 
	  e^{ \lambda \beta_\gamma^\lambda(s)}
	  \big( L\big(\gamma(s), \dot{\gamma}(s), 0\big) + c(H)  \big)\;ds
		  &\geq 
		  e^{\lambda \beta_\gamma^\lambda (b)}u_\lambda(\gamma(b)) \\
	  e^{\lambda\beta_\gamma^\lambda (a)}u_\lambda(\gamma(a))  
	  +
	  \int_{a}^{0} 
	  e^{ \lambda \beta_\gamma^\lambda(s)}
	  \big( L\big(\gamma(s), \dot{\gamma}(s), 0\big) + c(H)  \big)\;ds 
	  		&\geq u_\lambda(\gamma(0)).
\end{align*}
Since $\gamma$ minimizes \eqref{eq:dyn-pro-bisNegative}, we have
\begin{align*}
	u_\lambda(\gamma(0)) 
	  =  
	  e^{\lambda\beta_\gamma^\lambda (-t)}u_\lambda(\gamma(-t)) 
	  &
	  \int_{-t}^{0} 
	  e^{ \lambda \beta_\gamma^\lambda(s)}
	  \big( L\big(\gamma(s), \dot{\gamma}(s), 0\big) + c(H)  \big)\;ds.
\end{align*}
Summing the three equations, we conclude \eqref{eq:optimalityNegativeEXP}.

\medskip
{\bf Part (iii).} Repeating the argument in part (i) for \( \gamma \in \mathcal{M}_\lambda(x,t) \), where equality holds in \eqref{eq:timeMin}, shows that \( \gamma \) also minimizes \eqref{eq:dyn-pro-bisNegative}. For the other direction, if $\gamma \in \mathcal{C}^-(x,t;\R^n)$ that minimizes  \eqref{eq:dyn-pro-bisNegative}, then from \eqref{eq:optimalityNegativeEXP} we obtain
\begin{align*}
	\frac{d}{ds}\left(e^{\lambda \beta_{\gamma}^\lambda(s)} 
	u_\lambda (\gamma(s))   \right) 
	=
	e^{\lambda \beta_{\gamma}^\lambda(s)}
	\Big( L(\gamma(s), \dot{\gamma}(s), 0) +c(H)\Big)
	\qquad\text{for a.e.}\;s\in (-t,0).
\end{align*}
By \cite[Proposition 4.1]{mitake_asymptotic_2008}, 
$s\mapsto u_\lambda(\gamma(s))$ is absolutely continuous on $[-t,0]$, thus 
\begin{align*}
	\frac{d}{ds}\left(
	u_\lambda (\gamma(s))   \right)  
	- 
	\lambda K_\lambda(\gamma(s),\dot{\gamma}(s)) 
	\cdot 
	u_\lambda(\gamma(s)) 
	=
	L(\gamma(s), \dot{\gamma}(s), 0) +c(H)
	\qquad\text{for a.e.}\;s\in (-t,0).
\end{align*}
Hence, by the definition of $K_\lambda$ we obtain 
\begin{align*}
	\frac{d}{ds}\left(
	u_\lambda (\gamma(s))   \right)   
	= 
	L(\gamma(s), \dot{\gamma}(s), 
	\lambda u_\lambda(\gamma(s))) + c(H) 
		\qquad\text{for a.e.}\;s\in (-t,0).
\end{align*}
Taking integration we deduce that $\gamma\in \mathcal{M}_\lambda(x,t)$.
\medskip

{\bf Part (iv). } We observe that any $\gamma\in \mathcal{M}_\lambda(x,\infty)$ satisfies \eqref{eq:InfinityMinimizerExp}. 
\end{proof}

\begin{prop}\label{prop:BoundedMinimizerCompact} 
Assume \ref{itm:assumptions-h1}, \ref{itm:assumptions-h2} and either \ref{itm:assumptions-pp1}, \ref{itm:assumptions-pp2}, or \ref{itm:assumptions-pp3}.  
Fix $z\in \R^n$, there exists $\lambda_z\in(0,1)$ such that 
\begin{equation}\label{eq:BoundCurves}
	\sup_{\lambda \in (0,\lambda_z)} \sup  \left\lbrace \Vert \gamma\Vert_{L^\infty((-t,0])} : \gamma\in \mathcal{M}_\lambda(x, t), t>0 \right\rbrace < \infty. 
\end{equation} 
As a consequence, there exists \(\gamma \in \mathcal{M}_\lambda(z,\infty)\) such that
\begin{equation}\label{eq:boundz}
	|\dot{\gamma}(s)| +
	|\gamma(s)| \leq M_z \quad \text{for a.e. } s \leq 0.
\end{equation}
Additionally, under $\mathrm{(i)}$, we have $|\dot{\gamma}(s)| \leq M$ for a.e. $s\leq 0$ uniformly in $z$. If 
\eqref{eq:MonotoneWeakBelowH}
holds, we also have
\begin{align}\label{eq:UlambdaInfinity}
	u_\lambda(x) &= 
	\int_{-\infty}^0 e^{ \lambda \beta_{\gamma}^\lambda(s)} 
	\Big( 
		L\big(\gamma(s), \dot{\gamma}(s), 0\big) + c(H)  
	\Big)\;ds, \quad\text{where}\quad \beta^\lambda_\gamma(s) = \int_{s}^0 K_\lambda(\gamma(\tau ),\dot{\gamma}(\tau ))\;d\tau, s\leq 0.
\end{align}
\end{prop}

\begin{proof} From Proposition \ref{prop:VariationalFomular}, under either of these assumptions then $u_\lambda$ satisfies \eqref{eq:MinimizerTimePositive}. \medskip

Let $R_z = \max\{R_0, |z|\}$, where $R_0$ be the constant from \eqref{eq:growthL}, such that $L(x,v,0) + c(H) > \delta_0 |v|$ if $|x|\geq R_0$. From Corollary \ref{coro:UniformBoundedLocally}, there exists a constant $C_{z}$ such that $\max_{x\in \overline{B}_{R_z}(0)} u_\lambda(x) \leq C_{z}$. 
\begin{itemize}

\item If \ref{itm:assumptions-h3} holds, let $C$ be the constant from \eqref{eq:boundedVelocity} in Lemma \ref{lem:PropertiesUStrongCoerNew}, and $\overline{\kappa}=\overline{\kappa}_C$ be the constant such that $\partial_u L(x,v,u) \geq -\overline{\kappa}$ for all $|v|\leq C$.

\item If \ref{itm:assumptions-h4} holds, let $\overline{\kappa}$ be the corresponding constant such that $\partial_u L(x,v,u) \geq -\overline{\kappa}$.

\end{itemize}
We define
\begin{equation}\label{eq:lambdaz}
	\lambda_z = \min\left\lbrace \frac{1}{2}, \frac{\delta_0}{2\overline{\kappa}C_{z}} \right\rbrace. 
\end{equation}
Assume that \eqref{eq:BoundCurves} is false, then for some $\lambda \in (0,\lambda_z)$, there exist \( t_k \to \infty \) and \( \gamma_k \in \mathcal{M}_\lambda(z, t_k) \) such that
\begin{equation}\label{eq:assumeEscapeInfty}
	\lim_{k\to\infty} |\gamma_k(-t_k)| = +\infty.
\end{equation}
Let $\hat{t}_k \in [0,t_k]$ be the last time $\gamma_k(\cdot)$ remains in $\overline{B}_{R_z}(0)$, i.e.,
\begin{equation*}
	-\hat{t}_k = 
	\min \Big\lbrace -t: \gamma_k(-t) \in \partial B_{R_z}(0)  \Big\rbrace 
	\qquad\Longrightarrow\qquad
	\gamma_k([-t_k, -\hat{t}_k]) \notin \overline{B}_{R_z}(0). 
\end{equation*}
By \eqref{eq:optimalityNegativeEXP} of Proposition \ref{prop:exponentialNegative} we have
\begin{align*}
	&e^{ \lambda \beta^\lambda_{\gamma_k}(-\hat{t}_k)}
		u_\lambda(\gamma_k(-\hat{t}_k))  
	= 
	e^{ \lambda\beta^\lambda_{\gamma_k	}(-t_k)}
		u_\lambda(\gamma_k(-t_k)) 
	+ 
	\int_{-t_k}^{-\hat{t}_k} 
		e^{ \lambda 	\beta^\lambda_{\gamma_k	} (s)}
			\Big( L\big(\gamma_k(s), \dot{\gamma_k}(s), 0\big) + c(H)  \Big)
	\;ds .
\end{align*}
Using $u_\lambda \geq -C_0$ from Proposition \ref{prop:PerronLipschitzNoContact} we have
\begin{align}
	u_\lambda(\gamma_k(-\hat{t}_k))  
		&=
		e^{ \lambda 
		\left(  
		\beta^\lambda_{\gamma_k}(-t_k) 
		-
		\beta^\lambda_{\gamma_k}(-\hat{t}_k)
		\right) }  
		u_\lambda(\gamma_k(-t_k)) +
		\int_{-t_k}^{-\hat{t}_k} 
		e^{ 
		\lambda
		\left( 
		\beta^\lambda_{\gamma_k} (s)
		-
		\beta^\lambda_{\gamma_k} (-\hat{t}_k)
		\right)}
			\Big( L\big(\gamma_k(s), \dot{\gamma_k}(s), 0\big) + c(H)  \Big)
	\;ds  \nonumber\\
	&\geq -C_0 + \int_{\hat{t}_k-t_k}^{0} 
		e^{\lambda \int_s^{0} K_\lambda(\xi_k(\tau), \dot{\xi}_k(\tau))\;d\tau }
	\Big( L\big(\xi_k(s), \dot{\xi_k}(s), 0\big) + c(H)  \Big)
	\;ds \label{eq:estimateUlambda}
\end{align}
where \( \xi_k(s) = \gamma_k(s - \hat{t}_k) \) for \( s \in [\hat{t}_k - t_k, 0] \), using the fact that \( \beta^\lambda_{\gamma_k}(-t_k) - \beta^\lambda_{\gamma_k}(-\hat{t}_k) \leq 0 \).
 \medskip

By the choice of $C$ and $\overline{\kappa}$, we have
\begin{align*}
	 K_\lambda\left(\xi_k(s), \dot{\xi}_k(s)\right) \geq -\overline{\kappa}.
\end{align*}
The growth condition \eqref{eq:BoundLPositive} implies \( L(x,v,0) + c(H) \geq \delta_0 \) for \( |x| \geq R_0 \). Thus, by \eqref{eq:estimateUlambda}, we obtain
\begin{align*}
	C_{z}\geq u_\lambda(\gamma_k(-\hat{t}_k)) 
	\geq -C_0 + \delta_0 \int_{\hat{t}_k-t_k}^0 e^{\lambda \overline{\kappa} s}\;ds  = \frac{\delta_0\left(1-e^{-\lambda \overline{\kappa}|t_k-\hat{t}_k|}\right)}{\lambda \overline{\kappa}}. 	
\end{align*}
We deduce that, for $\lambda \in (0,\lambda_z)$ then
\begin{equation}\label{eq:estimateExpBelow}
	e^{-\lambda \overline{\kappa}|\hat{t}_k-t_k|} 
		\geq 1 - \frac{\lambda \overline{\kappa} C_{z}}{\delta_0} 
		\geq 1 - \frac{\lambda_z \overline{\kappa} C_{z}}{\delta_0} \geq \frac{1}{2} 
\end{equation} 
by the choice of $\lambda_z$ from \eqref{eq:lambdaz}. 
Using the the growth rate \eqref{eq:growthL} that $L(x,v,0) +c(H) \geq \delta_0|v|$ for $|x|\geq R_0$ in \eqref{eq:estimateUlambda}, we have
\begin{align*}
	C_{z} \geq u_\lambda(\gamma_k(-\hat{t}_k)) 
	&\geq -C_0 + \delta_0 \int_{\hat{t}_k-t_k}^0 e^{\lambda \overline{\kappa }s} |\dot{\xi}(s)|\;ds 	
		\geq -C_0 + \delta_0 e^{-\lambda \overline{\kappa}|\hat{t}_k-t_k|} \int_{\hat{t}_k-t_k}^0 |\dot{\xi}_k(s)|\;ds \\
	& \geq -C_0 +	\frac{\delta_0}{2} \left|\xi_k(0) - \xi_k(\hat{t}_k-t_k)\right| = -C_0 + \frac{\delta_0}{2}\left|\gamma_k(0) - \gamma_k(-t_k)\right|.
\end{align*}
We deduce that
\begin{align*}
	|\gamma_k(-t_k)| 
		\leq 
	|\gamma_k(0)| +  \frac{2\left(C_0 +  C_{z} \right)}{\delta_0} = |z| + \frac{2\left(C_0 +  C_{z} \right)}{\delta_0}. 
\end{align*}
This means \eqref{eq:assumeEscapeInfty} cannot happen if $\lambda < \lambda_z$, hence \eqref{eq:BoundCurves} must be true. \medskip

For \(\lambda \in (0, \lambda_z)\), let \(\gamma_k \in \mathcal{M}_\lambda(x, -k)\). For any $-k<b<a<0$, we have
\begin{align*}
	u_\lambda(\gamma_k(a)) - u_\lambda(\gamma_k(b))
	=
	\int_a^b \big(L(\gamma_k(s), \dot{\gamma}_k(s), \lambda u_\lambda(\gamma_k(s)))\big)\;ds
	\geq 
	\int_a^b \Big(L\left(\gamma_k(s), \dot{\gamma}_k(s), M_z\right)+c(H)\Big)\;ds
\end{align*}
thanks to the fact that $\gamma_k([-k,0])\subset \overline{B}_{M_z}(0)$. There must exists a constant $\widehat{C}_z$ such that $\Vert Du_\lambda\Vert_{L^\infty(\overline{B}_{M_z}(0))}\leq \widehat{C}_z$, as in Proposition \ref{prop:PerronLipschitzNoContact}. For $t\in  (-k,0)$ where $\dot{\gamma}_k(t)$ exists, we have
\begin{align*}
	L\left(\gamma_k(s), \dot{\gamma}_k(s), C_z\right)+c(H)\
		 \leq  
	 \widehat{C}_z\cdot|\dot{\gamma}_k(s)|
\end{align*} 
Recall that $|\gamma_k(s)| \leq M_z$ for all $s\in (-k,0)$, thanks to \eqref{eq:SuperLinearLtilde}, we deduce that $|\dot{\gamma}_k(s)| \leq \widehat{M}_z$ for a.e. $s\in (-k,0)$. We can apply a diagonal argument to obtain \eqref{eq:boundz}. 
We remark that in case (i), by \eqref{eq:boundedVelocity} in Lemma \ref{lem:PropertiesUStrongCoerNew}, the constant \(\widehat{M}_z\) is independent of \(z\).  \medskip

To see \eqref{eq:UlambdaInfinity}, we recall from Proposition \ref{prop:exponentialNegative} that for any $t>0$, as a minimizer we have
\begin{align}\label{eq:xyz}
	u_\lambda(z) 
	= 
	e^{\lambda \beta^\lambda_\gamma(-t)} 
	u_\lambda(\gamma(-t))
	+
	\int_{-t}^0 e^{\lambda \beta^\lambda_\gamma(s)}
	\Big(
		L
		\big(
		\gamma(s), \dot{\gamma}(s), 0)
		\big) + c(H)
	\Big)
	\;ds.
\end{align}
We have $u_\lambda(\gamma(s))$ bounded for all $s \leq 0$, since the minimizing path $\gamma$ satisfies $|\gamma(s)| \leq M_z$. The condition $\beta^\lambda_\gamma(s) \to -\infty$ as $s \to -\infty$ is required and is ensured by \eqref{eq:MonotoneWeakBelowH}, which is included in assumption \ref{itm:assumptions-h3}. Letting $t \to \infty$ in \eqref{eq:xyz}, and using the bound $L(x,v,0) \geq -c_0$ from \eqref{eq:BoundL}, together with the Monotone Convergence Theorem, we obtain \eqref{eq:UlambdaInfinity}.
\end{proof}

\section{Selection principle} \label{sec:selection-principle}
In this section we give the proof to Theorem \ref{thm:selection-principle}. 

\begin{lem}\label{lem:measureToBoundedFunctions} Let \(\gamma \in \mathrm{AC}((-\infty, 0]; \mathbb{R}^n)\) be such that $|\gamma(s)| + |\dot{\gamma}(s)| \leq R$ for a.e. $s\leq 0$. Define the probability Radon measure \(\mu^\lambda_\gamma \in \mathcal{R}^+(\mathbb{R}^n \times \mathbb{R}^n)\) such that
\begin{align}\label{eq:measure}
	\int_{\R^n\times\R^n} 
	\phi(x,v)\;d\mu^\lambda_\gamma(x,v) 
	= 
	\dfrac{\int_{-\infty}^0 
	e^{\lambda \beta^\lambda_\gamma(s)}
	\phi(\gamma(s), \dot{\gamma}(s))\;ds}
	{\int_{-\infty}^0 e^{\lambda \beta^\lambda_\gamma(s)}\;ds},
	 \qquad \phi\in C_c(\R^n\times \R^n),
\end{align}
where $\beta^\lambda_\gamma(s) = \int_s^0 K_\lambda(\gamma(s),\dot{\gamma}(s))\;ds$ for $s\leq 0$ as defined in \emph{Proposition \ref{prop:exponentialNegative}}, equation \eqref{eq:beta}. Then, the measure $\mu^\lambda_\gamma$ defined in \eqref{eq:measure} has compact support in $\R^n\times \R^n$. As a consequence, we have 
\begin{align}\label{eq:lem:measureToBoundedFunctions}
	\int_{\R^n\times\R^n} 
	\phi(x,v)\;d\mu^\lambda_\gamma(x,v) 
	= 
	\int_{\overline{B}_{R}(0)\times \overline{B}_{R}(0)} 
	\phi(x,v) \;d\mu^\lambda_\gamma(x,v), 
	\qquad\qquad  
	\phi \in C(\mathbb{R}^n \times \mathbb{R}^n). 
\end{align}
\end{lem}
\begin{proof} This follows directly from the definition that 
\begin{align*}
	&\int_{\R^n\times\R^n} 
	\phi(x,v)\;d\mu^\lambda_\gamma(x,v)  
	= 
	\dfrac{ 
		\int_{-\infty}^0 
		e^{\lambda \beta^\lambda_\gamma(s)}
		\phi(\gamma(s), \dot{\gamma}(s))\;ds 
	}
	{
		\int_{-\infty}^0 
		e^{\lambda \beta^\lambda_\gamma(s)}\;ds
	}
    = 
	\dfrac{ 
		\int_{-\infty}^0 
		e^{\lambda \beta^\lambda_\gamma(s)}
		\left(
            \phi\cdot \mathbf{1}_{\overline{B}_{R}(0)\times 
		\overline{B}_{R}(0)}
        \right)
		(\gamma(s), \dot{\gamma}(s))\;ds 
	}
	{
		\int_{-\infty}^0 
		e^{\lambda \beta^\lambda_\gamma(s)}\;ds
	} \\
	&\qquad\qquad\qquad 
    = \int_{\R^n\times\R^n} 
	\phi(x,v)\cdot \mathbf{1}_{\overline{B}_{R}(0)\times \overline{B}_{R}(0)}(x,v) \;d\mu^\lambda_\gamma(x,v) 
    = 
	\int_{\overline{B}_{R}(0)\times \overline{B}_{R}(0)} 
	\phi(x,v) \;d\mu^\lambda_\gamma(x,v) 
\end{align*}
for $\phi \in C_c(\mathbb{R}^n \times \mathbb{R}^n)$. Thus \(\mathrm{supp}(\mu^\lambda_\gamma) \subset \overline{B}_R(0) \times \overline{B}_R(0)\), and the result extends to all \(\phi \in C(\mathbb{R}^n \times \mathbb{R}^n)\). 
\end{proof}

\begin{defn}\label{defn:SMather} 
Assume \ref{itm:assumptions-h1}--\ref{itm:assumptions-h2}. 
For \( z \in \R^n \), let \( \mathfrak{M}(z,\lambda) \) be the set of discounted measures \( \mu^\lambda_\gamma \) from curves \( \gamma \in \mathcal{M}_\lambda(z, \infty) \) . We define \( \mathfrak{M}(z) \) as the set of all weak$^*$ limits of \( \mu^\lambda \) as \( \lambda \to 0^+ \) for fixed \( z \in \R^n \), and set  
\begin{equation}\label{eq:OurMatherSet}
    \mathfrak{M} := \bigcup_{z \in \R^n} \mathfrak{M}(z).
\end{equation}
\end{defn}

\begin{prop}[Discounted and Mather measures]\label{prop:supportRn}  
Assume \ref{itm:assumptions-h1}--\ref{itm:assumptions-h3}, and one of \ref{itm:assumptions-pp1}, \ref{itm:assumptions-pp2}, or \ref{itm:assumptions-pp3}.
We have the followings:
\begin{enumerate}
    \item[$\mathrm{(i)}$] $\mathfrak{M}\neq \emptyset$, and 
    \begin{align}\label{eq:holonomicRn}
        \int_{\R^n \times \R^n} v\cdot D\phi(x)\;d\mu(x,v) = 0 \qquad \text{for all}\; \phi\in C^1(\R^n), \mu \in \mathfrak{M}. 
    \end{align}
    \item[$\mathrm{(ii)}$] Let \( \mathcal{H} : \R^n \times \mathbb{R}^n \to \mathbb{R} \) satisfy 
    \ref{itm:assumptions-h1}--\ref{itm:assumptions-h2} 

    with the third argument fixed at \( u = 0 \). If \( u \in C(\R^n) \) is a viscosity subsolution of \( \mathcal{H}(x, Du) = c \) in \( R^n \), then the Legendre transform $\mathcal{L} = \mathcal{H}^*$ satisfies
    \begin{equation}\label{eq:minMeasuresOmegaRn}
    	\int_{\R^n\times\R^n}\mathcal{L}(x,v)\; d\mu(x,v)\geq -c \qquad\text{for all}\; \mu\in \mathfrak{M}. 
    \end{equation}
    \item[$\mathrm{(iii)}$] For all $\mu\in \mathfrak{M}$, we have
    \begin{align} \label{eq:minMeasuresRn}
        \int_{\R^n\times \R^n} L(x,v,0)\;d\mu(x,v) = -c(H) \quad\text{and}\quad 
        \int_{\R^n\times \R^n} L(x,v,\lambda u_\lambda(x))\;d\mu(x,v) 
        \geq -c(H). 
    \end{align}
\end{enumerate}

\end{prop}

\begin{proof} \qquad \medskip 

{\bf Part (i).} For any \( z \in \R^n \), Corollary \ref{cor:nonemtpyM} ensures \( \mathcal{M}_\lambda(z, \infty) \neq \emptyset \). Proposition \ref{prop:BoundedMinimizerCompact} ensures that for \( \gamma \in \mathcal{M}_\lambda(z,\infty) \), we have \( |\gamma(s)| + |\dot{\gamma}(s)| \leq M_z \) a.e. \( s \leq 0 \). Then, by Lemma \ref{lem:measureToBoundedFunctions}, \( \mathrm{supp}(\mu^\lambda) \subset \overline{B}_{M_z}(0) \times \overline{B}_{M_z}(0) \). Thus, along a subsequence \( \lambda_j \to 0 \), \( \mu^{\lambda_j} \rightharpoonup \mu \), and \( \mu \in \mathfrak{M}(z) \), so \( \mathfrak{M}(z) \neq \emptyset \). \medskip  

Take \( \phi \in C^1(\R^n) \). By Lemma \ref{lem:measureToBoundedFunctions}, we may integrate against the continuous function \( v \cdot D\phi(x) \in C(\R^n \times \R^n) \) and compute:
\begin{align*}
	&\int_{\R^n\times \R^n} v\cdot \nabla\phi(x)\;d\mu^\lambda(x,v) 
	= \frac
    {
        \int_{-\infty}^0 
        e^{\lambda\beta^\lambda_\gamma(s)}
        \cdot 
        \nabla\phi(\gamma(s))\cdot \dot{\gamma}(s)\;ds 
    }{
        \int_{-\infty}^0 e^{\lambda \beta^\lambda_\gamma(s)}\;ds 
    } \\
	&\qquad \qquad \qquad 
    = \frac
    {
        \int_{-\infty}^0 e^{\lambda	\beta^\lambda_\gamma(s)} \frac{d}{ds}\left(\phi(\gamma(s))\right) \;ds 
    }{
        \int_{-\infty}^0 e^{\lambda \beta^\lambda _\gamma(s)}\;ds 
    }
	= \frac{e^{\lambda \alpha^\lambda_\gamma(s)}\phi(\gamma(s))\big|_{-\infty}^0 -\int_{-\infty}^0 \phi(\gamma(s))\cdot \frac{d}{ds}\left(e^{\lambda\alpha_\gamma^\lambda(s)}\right)\;ds}{ \int_{-\infty}^0 e^{\lambda \alpha_\gamma^\lambda(s)}\;ds } .
\end{align*}
Thanks to \ref{itm:assumptions-h3}
and $|\gamma(s)| + |\dot{\gamma}(s)|\leq M_z$, we have $-\overline{\kappa}_z \leq K_\lambda(\gamma(s), \dot{\gamma}(s))\leq -\underline{\kappa}_z$ for a.e. $s\in (-\infty,0)$. Thus
\begin{align*}
	\lim_{s\to -\infty} 
    e^{\lambda \beta_\gamma^\lambda(s)}\phi(\gamma(s)) = 0, \qquad
    \frac{1}{\lambda \overline{\kappa}_z} \leq \int_{-\infty}^0 e^{\lambda \beta_\gamma^\lambda(s)}\;ds \leq \frac{1}{\lambda \underline{\kappa}_z},
    \qquad 
	e^{\lambda \beta_\gamma^\lambda(s)}\phi(\gamma(s))\Big|_{-\infty}^0 = \phi(\gamma(0)).
\end{align*}
Since $\frac{d}{ds}\left(e^{\lambda \beta^\lambda_\gamma(s)}\right) \geq 0$ for $s\leq 0$, we have
\begin{align}\label{eq:boundIPOmegaRn}
	\left|\int_{-\infty}^0  \phi(\gamma(s)) 
	\cdot 
	\frac{d}{ds}\left(e^{\lambda \beta_\gamma^\lambda(s)}\right)\;ds \right| 
	\leq 
	\Vert \phi\Vert_{L^\infty\left(\overline{B}_{M_z}(0)\right)}
	\int_{-\infty}^0 \left|\frac{d}{ds}
	\left(e^{\lambda \beta_\gamma^\lambda(s)}\right) \right|\;ds  
	= \Vert \phi\Vert_{L^\infty\left(\overline{B}_{M_z}(0)\right)} .
\end{align}
We deduce that 
\begin{align*}
	\left|\int_{\R^n\times \R^n} v\cdot \nabla\phi(x)\;d\mu^\lambda(x,v)\right|  \leq \lambda \overline{\kappa}  _z\cdot \left(|\phi(z)| + \Vert \phi\Vert_{L^\infty\left(\overline{B}_{M_z}(0)\right)} \right) .  
\end{align*} 
Letting \( \lambda \to 0^+ \), we obtain \eqref{eq:holonomicRn}, using the compact support of \( \mu^\lambda \).

\medskip

{\bf Part (ii).} For any \( \mu \in \mathfrak{M} \), there exists \( z \in \R^n \) and \( M_z > 0 \) with \( \mathrm{supp}(\mu) \subset \overline{B}_{M_z}(0) \times \overline{B}_{M_z}(0) \). If \( u \in C(\R^n) \) is a subsolution to \( \mathcal{H}(x, Du, 0) \leq c \), then \( u \) is locally Lipschitz by 
\eqref{eq:assumptions-LocalCoercive} in \ref{itm:assumptions-h1}.
Let \( R > 0 \) be large enough so that
\begin{equation*}
    \mathcal{A} \cup \overline{B}_{M_z}(0)\subset \Omega := B_R(0)  . 
\end{equation*}
Since \( u \in \mathrm{Lip}(\Omega) \), we can approximate it as in \cite[Lemma 2.10]{tu_generalized_2024}: for any \( \varepsilon > 0 \), there exists \( u^\varepsilon \in C^\infty(\R^n) \) with \( \mathcal{H}(x, Du^\varepsilon) \leq c + \mathcal{O}(\varepsilon) \) in \( \Omega \). By Fenchel--Young's inequality, we have
\begin{align*}
	v\cdot Du^\varepsilon(x) \leq \mathcal{L}(x,v) + \mathcal{H}(x,Du^\varepsilon(x)) \leq \mathcal{L}(x,v) + c + \mathcal{O}(\varepsilon),  \qquad  (x,v) \in \overline{\Omega} \times \R^n. 
\end{align*}
Integrating and applying Lemma \ref{eq:lem:measureToBoundedFunctions}, we get
\begin{align*}
    &\int_{\R^n\times \R^n} \Big(\mathcal{L}(x,v)+c\Big)\;d\mu(x,v)  
    = 
	\int_{\overline{\Omega}\times \R^n} \Big(\mathcal{L}(x,v)+c\Big)\;d\mu(x,v)   \\
	&\qquad\qquad\qquad \geq 
	\int_{\overline{\Omega}\times \R^n} v\cdot Du^\varepsilon(x)\;d \mu(x,v) - \mathcal{O}(\varepsilon) 
    = \int_{\R^n\times \R^n} v\cdot Du^\varepsilon(x)\;d \mu(x,v) - \mathcal{O}(\varepsilon) = - \mathcal{O}(\varepsilon)
\end{align*}
thanks to \eqref{eq:holonomicRn}. Letting \( \varepsilon \to 0^+ \) yields \eqref{eq:minMeasuresOmegaRn}. \medskip

{\bf Part (iii).} The inequality cases follow by applying \( \mathcal{H}(x,p) = H(x,p,0) \) and \( \mathcal{H}(x,p) = H(x,p,\lambda u_\lambda(x)) \), both satisfying the assumptions by Lemma \ref{lem:HTilde}. As a minimizer, we have
\begin{align*}
	\frac{d}{ds} u(\gamma(s)) = L(\gamma(s), \dot{\gamma}(s), \lambda u_\lambda(\gamma(s))) + c(H) \qquad\text{for a.e.}\;s<0
\end{align*}
Therefore, as in \eqref{eq:boundIPOmegaRn}, using \( \frac{d}{ds}(e^{-\lambda\beta_\gamma^\lambda(s)}) \geq 0 \), \( u_\lambda \) in place of \( \phi \), and integrating against continuous functions via Lemma \ref{lem:measureToBoundedFunctions}, we obtain:
\begin{align*}
	&\left|\int_{\R^n\times\R^n} \Big(L(x,v,\lambda u_\lambda(x)) +c(H)\Big)\;d\mu^\lambda(x,v)\right| = \left|\frac{\int_{-\infty}^0 e^{\lambda \beta_\gamma^\lambda(s)} 
	\big(L(\gamma(s), \dot{\gamma}(s), \lambda u_\lambda(\gamma(s))) + c(H)\big)\;ds }
	{
	\int_{-\infty}^0 e^{\lambda \beta_\gamma^\lambda(s)}\;ds
	}\right| \\
	&\qquad = 
	\left|\frac{
	\int_{-\infty}^0 e^{\lambda\beta_\gamma^\lambda(s)}\frac{d}{ds}\left(u_\lambda(\gamma(s))\right)\;ds
	}{
	\int_{-\infty}^0 e^{\lambda \beta_\gamma^\lambda(s)}\;ds
	} \right|
	= \left| \frac{
	u_\lambda(0) - \int_{-\infty}^0 u_\lambda(\gamma(s))\frac{d}{ds}
	\left(
	e^{\lambda\beta_\gamma^\lambda(s)}
	\right)\;ds
	}{
	\int_{-\infty}^0 e^{\lambda \beta_\gamma^\lambda(s)}\;ds
	} \right| \leq 2\lambda\overline{\kappa} \cdot 
	\Vert u_\lambda\Vert_{L^\infty(\overline{B}_{M_z}(0))}.
\end{align*}
Since \( L(x,v,\lambda u_\lambda(x)) \to L(x,v,0) \) uniformly on \( \overline{B}_{M_z}(0) \times \overline{B}_{M_z}(0) \), the common support of all \( \mu^\lambda \) for \( \lambda \in (0,\lambda_z) \), we obtain the first equality in \eqref{eq:minMeasuresRn} as \( \lambda \to 0^+ \).
\end{proof}

\begin{rmk} It is standard to define the set of \emph{holonomic} (or \emph{closed}) probability measures \( \mathcal{C}(\R^n \times \R^n) \subset \mathcal{P}(\R^n \times \R^n) \) as those satisfying  
\[
\int_{\R^n\times\R^n} |v|\, d\mu(x,v) < \infty \quad \text{and} \quad \int_{\R^n\times\R^n} v \cdot D\phi(x)\, d\mu(x,v) = 0 \quad \text{for all } \phi \in C^1(\R^n).
\]
In bounded domains, Mather measures are closed measures minimizing \(\langle \mu, L \rangle = -c(H)\). Our set \(\mathfrak{M}\) from \emph{Definition \ref{defn:SMather}} resembles this property, though possibly smaller than the standard set of measures. A more detailed treatment, including connections to the functional analysis approach in \cite{ishii_asymptotic_2008}, is left for future work.
\end{rmk}

\begin{proof}[Proof of Theorem \ref{thm:selection-principle}] Let us define
\begin{equation*}
	\mathcal{E} = \left\lbrace
	w\in C(\R^n): H(x,Dw(x), 0) \leq c(H)\;\text{in}\;\R^n: 
	\int_{\R^n\times\R^n} w(x)\cdot\partial_u L(x,v,0)\;d\mu(x,v) \geq 0\;\forall\; 
	\mu \in \mathfrak{M}
	\right\rbrace.
\end{equation*}
By Proposition \ref{prop:PerronLipschitzNoContact}, the family \( \{u_\lambda\}_{\lambda \in (0,1)} \) is locally bounded and equi-Lipschitz, so up to a subsequence \( \lambda_i \to 0 \), we have \( u_{\lambda_i} \to u \in C(\mathbb{R}^n) \) locally uniformly. By stability, \( u \) solves \eqref{eq:E}. To verify that \( u \in \mathcal{E} \), take any \( \mu \in \mathfrak{M}\). Then by \eqref{eq:minMeasuresRn} in Proposition \ref{prop:supportRn} we have
\begin{equation*}
	\int_{\R^n\times\R^n} \frac{L(x,v,\lambda u_\lambda(x))-L(x,v,0)}{\lambda} \; d\mu (x,v)\geq 0, \qquad \mu \in \mathfrak{M}.
\end{equation*}
By definition of $\mathfrak{M}$, every $\mu$ has compact support. Let $\lambda = \lambda_i \to 0^+$, using $u_{\lambda_i}\to u$ uniformly on compact sets, we obtain $\int_{\R^n\times\R^n} \partial_u L(x,v,0)\cdot u (x)\; d\mu  (x,v)\geq 0$, thus $u\in \mathcal{E}$. 
\medskip

To show \( u = \sup \mathcal{E} \), take any \( w \in \mathcal{E} \) and \( z \in \mathbb{R}^n \). We aim to prove \( u(z) \geq w(z) \). By Proposition \ref{prop:BoundedMinimizerCompact}, there exists \( \gamma \in \mathcal{M}_\lambda(z, \infty) \) with \( \gamma(0) = z \) and \( |\gamma(s)| + |\dot{\gamma}(s)| \leq M_z \) for a.e. \( s \leq 0 \), and
\begin{align}\label{eq:minimizeULambda}
	u_\lambda(z) 
	= 
	e^{\lambda \beta^\lambda_\gamma(-t)} 
	u_\lambda(\gamma(-t)) 
	+ 
	\int_{-t}^0 
	e^{\lambda \beta^\lambda_\gamma(s)}
	\Big(L(\gamma(s), \dot{\gamma}(s),0) + c(H)\Big)
	\;ds 
	\qquad \text{for all}\;t>0. 
\end{align}
Thanks to \eqref{eq:assumptions-LocalCoercive} in \ref{itm:assumptions-h1},
\( w \) is locally bounded and locally Lipschitz. By \cite[Proposition 4.1]{mitake_asymptotic_2008}, for every $t>0$, the map $s\mapsto w(\gamma(s))$ is absolutely continuous on $(-t,0)$, and there exists a function $p\in L^\infty((-t,0);\R^n)$ such that 
\begin{equation*}
	\frac{d}{ds}\big(w(\gamma(s))\big) = p(s)\cdot \dot{\gamma}(s) \qquad\text{and}\qquad p(s) \in \partial_c w(\gamma(s)) \qquad\text{a.e.}\;s\in (-t,0),
\end{equation*}
where $\partial_c w$ is the Clarke differential of $w$. We observe that $H(\gamma(s), p(s), 0) \leq c(H)$ for a.e. $s\leq 0$. Thus for a.e. $s\leq 0$ we have
\begin{align}\label{eq:fenchelOmega}
	L(\gamma(s), \dot{\gamma}(s),0) + c(H) 
	& \geq L(\gamma(s), \dot{\gamma}(s),0) + H\left(\gamma(s), p(s), 0\right) \geq  p(s)\cdot\dot{\gamma}(s) = \frac{d}{ds}\left(w(\gamma(s))\right)  .
\end{align}
Therefore, from \eqref{eq:minimizeULambda} and \eqref{eq:fenchelOmega} we have
\begin{align*}
	\frac{d}{ds}
	\left(
		e^{\lambda \beta^\lambda_\gamma(s)}
		u_\lambda(\gamma(s))
	\right) 
	&= 
	e^{\lambda \beta^\lambda_\gamma(s)} 
		\Big(L(\gamma(s), \dot{\gamma}(s),0) + c(H)\Big) 		
	\geq e^{\lambda \beta^\lambda_\gamma(s)} 
		\frac{d}{ds}\left(w(\gamma(s))\right)  
			&&\qquad\text{for a.e.}\;s \leq 0.
\end{align*}
Using integration by parts for the absolutely continuous function $s\mapsto w(\gamma(s))$, we have 
\begin{align*}
	\int_{-t}^0 
	\frac{d}{ds} 
	\left(e^{\lambda \beta^\lambda_\gamma(s)}
	u_\lambda(\gamma(s))\right)\;ds 
	&\geq 
	\int_{-t}^0 e^{\lambda \beta^\lambda_\gamma(s)} 
		\frac{d}{ds}\left(w(\gamma(s))\right)  \;ds \\
	&= \left(e^{\lambda \beta^\lambda_\gamma(s)} w(\gamma(s)) \right)\Big|_{s=-t}^{s=0} 
	- \int_{-t}^0 w(\gamma(s)) \frac{d}{ds}\left(
	e^{\lambda \beta^\lambda_\gamma(s)} \right) \;ds. 
\end{align*}
Therefore we obtain
\begin{align}
	u_\lambda(z) 
	- 
	e^{\lambda \beta^\lambda_\gamma(-t)} 
	u_\lambda(\gamma(-t)) 
	&\geq w(z) 
	- 
	e^{\lambda \beta^\lambda_\gamma(s)} w(\gamma(-t)) 
	+ \lambda \int_{-t}^0 
	e^{\lambda \beta^\lambda_\gamma(s)} 
	w(\gamma(s)) \cdot K_\lambda\big(\gamma(s), \dot{\gamma}(s)\big) \;ds.  \label{eq:estwaRn}
\end{align}
Since $|\gamma(s)| + |\dot{\gamma}(s)| \leq M_z$ for a.e. $s\leq 0$, from 
\ref{itm:assumptions-h3} 
we have $K_\lambda(\gamma(s), \dot{\gamma}(s)) \leq -\underline{\kappa}$ for a.e. $s\leq 0$, hence $\beta^\lambda_\gamma(s) \leq \underline{\kappa}s$ for all $s\leq 0$, which means (we note that $\gamma$ depends on $\lambda$ as well)
\begin{align}\label{eq:liminfPosRn}
	0\leq \int_{-\infty}^0 
	e^{\lambda \beta_\gamma^\lambda(s)}\;ds 
	\leq \int_{-\infty}^0 
	e^{-\lambda \underline{\kappa} s}\;ds  
	= \frac{1}{\lambda \underline{\kappa}}
	 \qquad\Longrightarrow\qquad 
	 \frac{1}{\underline{\kappa}} 
	 \geq 
	 \liminf_{\lambda\to 0^+}  
	 \int_{-\infty}^0\lambda  
	 e^{\lambda \beta_\gamma^\lambda(s)}\;ds\geq 0. 
\end{align}
Since $|\gamma(s)|\leq M_z$, we have \( |u_\lambda(\gamma(s))|\leq C_z \) for a.e. $s\leq 0$ thanks to Proposition \ref{prop:PerronLipschitzNoContact}. Using \eqref{eq:assumptions-DiffC1} in \ref{itm:assumptions-h3} 
and the fact that we can ass assume super linearlity without loss of generality as in Proposition \ref{prop:SuperlinearReduction}, there is a modulus $\omega_z$ such that
\begin{align*}
	\left| K_\lambda(x,v) - \partial_uL(x,v,0)\right| 
	= 
	\left|
	\frac{L(x,v,\lambda u_\lambda(x)) - L(x,v,0)}{\lambda u_\lambda(x)} 
	- \partial_u L(x,v,0) \right| 
	\leq \omega( \lambda C_z). 
\end{align*}
From \eqref{eq:estwaRn} we have
\begin{align}
	&u_\lambda(z) 
	- 
	e^{\lambda \beta^\lambda_\gamma(-t)} 
	u_\lambda(\gamma(-t)) 
	\geq w(z) 
	- 
	e^{\lambda \beta^\lambda_\gamma(s)} w(\gamma(-t)) \nonumber  \\
	&\qquad\qquad
	+ \lambda \int_{-t}^0 
	e^{\lambda \beta^\lambda_\gamma(s)} 
	w(\gamma(s)) \cdot \partial_uL(\gamma(s), \dot{\gamma}(s),0\big) \;ds 
	- \omega(\lambda C_z) \cdot \lambda \int_{-t}^0 
	e^{\lambda \beta^\lambda_\gamma(s)} 
	w(\gamma(s))\;ds .
	\label{eq:estwaRn2}
\end{align} 
Let \( t \to \infty \) in \eqref{eq:estwaRn2}, thanks to \eqref{eq:liminfPosRn}, $\sup_{s\leq 0} |w(\gamma(s))| \leq \Vert w\Vert_{L^\infty(B_{M_z}(0))}$, and the Dominated Convergence Theorem we have
\begin{align}
	u_\lambda(z) 
	&\geq w(z) + \lambda \int_{-\infty}^0  
	e^{\lambda \beta^\lambda_\gamma(s)} 
	w(\gamma(s))
		\cdot  
	\partial_u L(\gamma(s),\dot{\gamma}(s),0)\;ds \nonumber \\
	&= w(z) + 
	\lambda \left(\int_{-\infty}^0 
	e^{\lambda \beta^\lambda_\gamma(s)}\;ds \right)
	\left(
		\int_{\R^n\times \R^n} 
		w(x)
			\cdot
		\partial_u L(x,v,0)\;d\mu^\lambda_\gamma(x,v)
	\right). 
	\label{eq:est2wx}
\end{align}
where we invoke \eqref{eq:lem:measureToBoundedFunctions} from Lemma \ref{lem:measureToBoundedFunctions}. Together with the fact that along a subsequence $\mu^\lambda_\gamma\rightharpoonup \mu$ and $\mu \in \mathfrak{M}$, we deduce from \eqref{eq:liminfPosRn} and \eqref{eq:est2wx} that, as $\lambda_i\to 0$ then
\begin{align*}
	u(z) \geq w(z) + 
	\left(
	\liminf_{\lambda\to 0^+} \int_{-\infty}^0  \lambda  e^{\lambda \alpha_\gamma^\alpha(s)}\;ds
	\right)
	\cdot 
	\left(
	\int_{\R^n\times \R^n} w(x)\cdot\partial_uL(x,v)\;d\mu(x,v)
	\right) 
	\geq w(z)
\end{align*}
since $w\in \mathcal{E}$. This completes the proof.
\end{proof}

\section{Localization with state-constraint solution}\label{sec:localization}

We emphasize that, unlike the case of discounted Hamiltonians considered in \cite{ishii_vanishing_2020}, it is not immediate that the solution \( u_\lambda \) coincides with a state constraint solution \( \vartheta_\lambda \) in \( B_{M_z}(0) \), as in \eqref{eq:StateConstraintExpInfityLipschitz}. We recall the formula \eqref{eq:UlambdaInfinity} for $u_\lambda$, and \eqref{eq:InfiniteExpOptimal} for $\vartheta$ here for convenience, using the curve $\gamma\in \mathcal{M}_\lambda(x,\infty)$ in \eqref{eq:boundz}: 
\begin{align*}
	u_\lambda(z) 
	= 
	\int_{-\infty}^0 
	e^{\lambda \beta^\lambda_\gamma(s)} 
	\big(L(\gamma(s), \dot{\gamma}(s), 0) + c(H)\big)\;ds 
	\quad  \text{and} \quad 
	\vartheta_\lambda(z) 
	= 
	\int_{-\infty}^0 
	e^{\lambda \alpha^\lambda_\gamma(s)} 
	\big(L(\gamma(s), \dot{\gamma}(s), 0) + c(H)\big)\;ds. 
\end{align*}
The difficulty lies in the fact that the exponential weights \( \beta^\lambda_\gamma \) and \( \alpha^\lambda_\gamma \) differ and depend on the solutions, unlike the discounted case \( H(x,p,u) = H(x,p) + u \) where both equal \( \lambda s \). Thus, establishing the localization \( u_\lambda(z) = \vartheta_\lambda(z) \) is nontrivial and more delicate.
Nevertheless, by suitably modifying the weight in the exponential formulas \eqref{eq:UlambdaInfinity} and \eqref{eq:InfiniteExpOptimal}, we can still recover the localization \( u_\lambda(z) = \vartheta_\lambda(z) \) at a fixed point. We start by defining a different discounted index. 

\begin{defn} \label{defn:NewIndices}
Assume \ref{itm:assumptions-h1}, \ref{itm:assumptions-h2}. 
Let $C_0$ be the constant such that the maximal solution $u_\lambda$ to \eqref{eq:DP} satisfies $\inf_{\R^n} u_\lambda \geq -C_0$. For $(x,v)\in \R^n\times\R^n$, we define
\begin{equation}\label{eq:DiscountedIndexNew}
    \mathbf{K}_\lambda(x,v) = 
    \begin{cases}
    \partial_u^+ L(x,v,0) &\text{if}\; u_\lambda (x) = 0, \vspace{0.1cm}\\ 
    \dfrac{L(x,v,\lambda u_\lambda(x)) - L(x,v,-\lambda C_0)}{\lambda u_\lambda(x) - (-\lambda C_0)} &\text{if}\; u_\lambda (x) \neq 0.
    \end{cases}
\end{equation}
Let $\Omega\subset\R^n$ be an open, bounded domain with $C^2$ boundary that contains $\mathcal{A}$, so that $c_\Omega(H) = c(H)$ by \emph{Lemma \ref{lem:ErgodicConstantTheSame}}. Let $\vartheta_\lambda$ be the state-constraint solution to \eqref{eq:propStateConstraintBoundOmega}, we define 
\begin{equation}\label{eq:DiscountedIndexOmegaNew}
	\mathbf{k}_\lambda(x,v) = 
    \begin{cases}
    \partial_u^+ L(x,v,0) &\text{if}\; \vartheta_\lambda (x) = 0, \vspace{0.1cm}\\ 
    \dfrac{L(x,v,\lambda \vartheta_\lambda(x)) - L(x,v,-\lambda C_0)}{\lambda \vartheta_\lambda(x) - (-\lambda C_0)} &\text{if}\; \vartheta_\lambda (x) \neq 0.
    \end{cases}
\end{equation}
\end{defn}
By comparison principle we have $\vartheta_\lambda (x) \geq u_\lambda(x) \geq -C_0$ for $x\in \R^n$. Since $L(x,v,\lambda u_\lambda(x)) \leq L(x,v,\lambda \vartheta _\lambda(x))$, we have
\begin{align*}
	\mathbf{K}_\lambda(x,v) &= 
	\dfrac{
		L(x,v,\lambda u_\lambda(x)) - L(x,v,- \lambda C_0)
	}{
	\lambda u_\lambda(x) - (- \lambda C_0)
	} \\
	& \geq 
	\dfrac{
		L(x,v,\lambda \vartheta _\lambda(x)) - L(x,v,- \lambda C_0)
	}{
	\lambda u_\lambda(x) - (- \lambda C_0)
	}  
	\geq
	\dfrac{
		L(x,v,\lambda \vartheta _\lambda(x)) - L(x,v,- \lambda C_0)
	}{
	\lambda \vartheta _\lambda(x) - (- \lambda C_0)
	}  
	=
	\mathbf{k}_\lambda(x,v).
\end{align*}

\begin{prop}\label{prop:LocalExponentialNegative} 
Assume \ref{itm:assumptions-h1}, \ref{itm:assumptions-h2} 
and that $u_\lambda$ satisfies \eqref{eq:MinimizerTimePositive}. 
\begin{itemize}
\item[$\mathrm{(i)}$] For $(x,t)\in \overline{\Omega}\times (0,\infty)$, we have 
\begin{align*}
	\vartheta_\lambda(x) + C_0 
	= 
	\inf_{\gamma\in \mathcal{C}^-(x,t;\overline{\Omega})} 
	\left\lbrace 
	e^{\lambda \mathbf{A}^\lambda_\gamma(-t) 
	}\vartheta_\lambda(\gamma(-t))  
	+ \int_{-t}^0 
	e^{\lambda\mathbf{A}_{\gamma}^\lambda (s)} 
	\Big(
		L(\gamma(s), \dot{\gamma}(s), -\lambda C_0) + c(H)
	\Big)\;ds
	\right\rbrace,
\end{align*}
where 
\begin{equation*}
	\mathbf{A}_\gamma^\lambda(s) = \int_{s}^0 
	\mathbf{k}_{\lambda}(\gamma(\tau), \dot{\gamma}(\tau))
	\;d\tau \qquad 
	\text{for}\;s\in [-t,0]. 
\end{equation*} 
\item[$\mathrm{(ii)}$]
If \eqref{eq:MonotoneWeakBelowH} holds, then 
\begin{align}
	\vartheta_\lambda(x) + C_0
	&= \min_{\gamma(0) = x} 
	\left\lbrace 
	\int_{-\infty}^0 
	e^{\lambda\mathbf{A}_{\gamma}^\lambda(s)}
	\Big(
		L(\gamma(s), \dot{\gamma}(s), -\lambda C_0) 
		+ c_\Omega(H)
	\Big)\;ds: 
	\gamma\in \mathrm{Lip}
	\left((-\infty,0);\overline{\Omega}\right)  
	\right\rbrace. \label{eq:StateConstraintExpInfityLipschitzNew}
\end{align}

\item[$\mathrm{(iii)}$] If $\gamma\in \mathcal{M}_\lambda(x,\infty)$, then 
\begin{equation}\label{eq:Local}
	u_\lambda(x) + C_0
	=
	e^{ \lambda \mathbf{B}_{\gamma}^\lambda(-t)} u_\lambda(\gamma(-t)) 
	+ 
	\int_{-t}^0 
	e^{ \lambda \mathbf{B}_\gamma^\lambda(s)} 
		\Big( 
			L\big(\gamma(s), \dot{\gamma}(s), -\lambda C_0\big) + c(H)  
		\Big)\;ds 
\end{equation}
for all $t>0$,  where 
\begin{equation*}
	\mathbf{B}_{\gamma}^\lambda(s) = \int_{s}^0  \mathbf{K}_\lambda(\gamma(\tau), \dot{\gamma}(\tau))\;d\tau, \qquad s\in [-t,0].
\end{equation*}
\end{itemize}
\end{prop}

\begin{proof} Take $\gamma\in \mathcal{C}^-(x,t;\overline{\Omega})$, by Proposition \ref{prop:curve} we have for $-t<b<a<0$ that
\begin{equation}\label{eq:timeMinNewOmega}
	\vartheta_\lambda(\gamma(a)) 
	\leq 
	\vartheta_\lambda(\gamma(b)) 
	+
	\int_{b}^a
	\Big( 
		L\left((\gamma(s), \dot{\gamma}(s), \lambda 
		\vartheta_\lambda(\gamma(s))\right) 
	+ c(H)
	\Big)\;ds.
\end{equation}
By \cite[Proposition 4.1]{mitake_asymptotic_2008}, 
$s\mapsto \vartheta_\lambda(\gamma(s))$ is absolutely continuous on $[-t,0]$, thus 
\begin{align*}
	\frac{d}{ds}\big(\vartheta_\lambda(\gamma(s))\big) 
	\leq 
	L
	\big(
	\gamma(s), \dot{\gamma}(s), 
	\lambda \vartheta_\lambda (\gamma(s))
	\big)
	 + c(H) 
	\qquad\text{for a.e.}\;s\in (-t,0). 
\end{align*}
For a.e. $s\in (-t,0)$ we have
\begin{align*}
	&\frac{d}{ds}
	\left(
		e^{\lambda \mathbf{A}_{\gamma}^\lambda(s)} 
	\big(
		\vartheta_\lambda (\gamma(s))   + C_0
	\big)
	\right) 
	= 
		e^{\lambda \mathbf{A}_{\gamma}^\lambda(s)}
	\left(
		-\lambda 
		\big(
			  \vartheta_\lambda(\gamma(s)) 
			+ C_0
		\big) \cdot 
		\mathbf{k}_{\lambda}(\gamma(s), \dot{\gamma}(s))
		+ 
		\frac{d}{ds}
		\big(\vartheta_\lambda(\gamma(s)) + C_0\big) 
	\right) \\
	&\qquad\qquad
	\leq 
	e^{\lambda \mathbf{A}_{\gamma}^\lambda(s)}
	\Big(
		L(\gamma(s), \dot{\gamma}(s), -\lambda C_0) 
		- 
		L(\gamma(s), \dot{\gamma}(s),\lambda \vartheta_\lambda(\gamma(s))) 
		+ 
		L\big(
		\gamma(s), \dot{\gamma}(s), 
		\lambda \vartheta_\lambda(\gamma(s))
		\big) + c(H)
	\Big) \\
	&\qquad\qquad 
	= e^{\lambda \mathbf{A}_{\gamma}^\lambda(s)}
	\Big(
		L(\gamma(s), \dot{\gamma}(s), -\lambda C_0)  + c(H)
	\Big). 
\end{align*}
Integrating over $[-t,0]$, we obtain
\begin{align*}
	\vartheta_\lambda(x) + C_0 
	&= 
	e^{\lambda \mathbf{A}_{\gamma}^\lambda(0)}
	\Big(\vartheta_\lambda(\gamma(0))  + C_0\Big) \\
	&\leq 
	e^{\lambda \mathbf{A}_{\gamma}^\lambda(-t)}
	\Big(\vartheta_\lambda(\gamma(-t))  + C_0\Big)
	+ 
	\int_{-t}^0 e^{\lambda\mathbf{A}_{\gamma}^\lambda(s)} 
	\Big(L(\gamma(s), \dot{\gamma}(s), -\lambda C_0) + c(H)\Big)\;ds. 
\end{align*}
Since $\gamma \in \mathcal{C}^-(x,t;\overline{\Omega})$ is chosen arbitrarily, we obtain \eqref{eq:OptimalControlStateConstraintExp}. 
\medskip

For \eqref{eq:StateConstraintExpInfityLipschitzNew}, since \( |\dot{\gamma}| \leq C \), assumption 
\eqref{eq:MonotoneWeakBelowH}
implies \( \mathbf{k}_\lambda(\gamma(s), \dot{\gamma}(s)) \leq -\underline{\kappa}_C \). Letting \( t \to \infty \) in \eqref{eq:OptimalControlStateConstraintExp} gives the result. For \eqref{eq:Local}, take \( \gamma \in \mathcal{M}_\lambda(x,\infty) \) and repeat the argument using equality in \eqref{eq:timeMinNewOmega} to conclude. 
\end{proof}

\subsection{Proof of Theorem \ref{thm:localization}}

\begin{proof}[Proof of Theorem \ref{thm:localization}] Let $\lambda_z$ be the constant from Proposition \ref{prop:BoundedMinimizerCompact}, we can define
\begin{align*}
	M_z: = \sup_{\lambda \in (0,\lambda_z)}\sup  \left\lbrace \Vert \gamma\Vert_{L^\infty((-t,0])} : \gamma\in \mathcal{M}_\lambda(x, t), t>0 \right\rbrace < \infty. 
\end{align*}
We redefine \( M_z \), if needed, so that \( \mathcal{A} \subset \Omega := B_{M_z}(0) \), ensuring \( c_\Omega(H) = c(H) \) by Lemma \ref{lem:ErgodicConstantTheSame}. Take $\gamma:(-\infty,0]\to \R^n$ be the curve as in \eqref{eq:boundz}. Since $|\dot{\gamma}(s)|\leq M_z$ for a.e. $s\in (-\infty,0]$, from 
\eqref{eq:MonotoneWeakBelowH}, which is guaranteed by \ref{itm:assumptions-h3}  
we have 
\begin{equation*}
	\mathbf{K}_\lambda(\gamma(s),\dot{\gamma}(s)) \leq -\underline{\kappa}_z
	\qquad\Longrightarrow\qquad
	\mathbf{B}^\lambda_\gamma(s) \leq -\underline{\kappa}_zs \qquad\text{for a.e.}\;s\in (-\infty,0]. 
\end{equation*}
Since $|\gamma(z)|\leq M_z$, we have $u_\lambda(\gamma(z)) \leq C_z$, thus we can let $t\to \infty$ in \eqref{eq:Local} to obtain that 
\begin{align*}
	u_\lambda(z) + C_0
	&= \int_{-\infty}^0
	e ^{\lambda \mathbf{B}^\lambda_\gamma(s)}
	\Big(L(\gamma(s), \dot{\gamma}(s), -\lambda C_0) + c(H)\Big)
	\;ds  \\ 
	&\geq 
	\inf_{\gamma(0)=z} 
	\left\lbrace
	\int_{-\infty}^0
	e ^{\lambda \mathbf{B}^\lambda_\gamma(s)}
	\Big(L(\gamma(s), \dot{\gamma}(s), 0) + c(H)\Big)
	\;ds:
	\gamma\in \mathrm{Lip}
	\left((-\infty, 0];\overline{B}_{M_z}(0)\right)
	\right\rbrace = \vartheta_{\lambda, M_z}(z)
\end{align*}
We note that
\begin{align*}
	\mathbf{K}_\lambda(x,v) \geq \mathbf{k}_\lambda(x,v)
	\qquad\Longrightarrow\qquad
	\mathbf{B}_\lambda(x,v) \geq  \mathbf{A}_\lambda(x,v).
\end{align*}
Hence, we obtain that 
\begin{align*}
	u_\lambda(z) + C_0 
	&\geq 
	\inf_{\gamma(0)=z} 
	\left\lbrace
	\int_{-\infty}^0
	e ^{\lambda \mathbf{B}^\lambda_\gamma(s)}
	\Big(L(\gamma(s), \dot{\gamma}(s), 0) + c(H)\Big)
	\;ds:
	\gamma\in \mathrm{Lip}
	\left((-\infty, 0];\overline{B}_{M_z}(0)\right)
	\right\rbrace\\
	&\geq 
	\inf_{\gamma(0)=z} 
	\left\lbrace
	\int_{-\infty}^0
	e ^{\lambda \mathbf{A}^\lambda_\gamma(s)}
	\Big(L(\gamma(s), \dot{\gamma}(s), 0) + c(H)\Big)
	\;ds:
	\gamma\in \mathrm{Lip}
	\left((-\infty, 0];\overline{B}_{M_z}(0)\right)
	\right\rbrace	\\
	&= \vartheta_{\lambda, M_z}(z) + C_0
\end{align*}
by \eqref{eq:StateConstraintExpInfityLipschitzNew} of Proposition \ref{prop:LocalExponentialNegative}. On the other hand, it is clear that $u_\lambda \leq \vartheta_{\lambda, M_z}$, thus we obtain the conclusion $u_\lambda(z) = \vartheta_{\lambda, M_z}(z)$ for all $\lambda \in (0,\lambda_z)$.
\end{proof}

\subsection{An alternative proof of Theorem \ref{thm:selection-principle}} \label{subsection:alternative-thm}

By Proposition \ref{prop:PerronLipschitzNoContact}, the family \(\{u_\lambda\}_{\lambda \in (0,1)}\) is equibounded and equi-Lipschitz on compact subsets of \(\mathbb{R}^n\). Hence, by Arzel\`a--Ascoli Theorem, any sequence has a locally uniformly convergent subsequence, and the set of such limits is nonempty. We now show the limit is unique. Proposition \ref{prop:PerronLipschitzNoContact} ensures the existence of \(\lambda_z > 0\) and \(R_z > 0\) such that \(\mathcal{A} \subset B_{R_z}(0)\), and
\begin{equation*}
	u_\lambda(z) = \vartheta_{\lambda}(z) \qquad\text{for all}\; \lambda \in (0,\lambda_z),
\end{equation*}
where $\vartheta_{\lambda}$ is the state-constraint to \eqref{eq:propStateConstraintBoundOmega} in $\Omega = B_{R_z}(0)$. By Theorem \ref{thm:VanishingOnStateConstraintBounded} we have
\begin{align*}
	\lim_{\lambda\to 0^+} u_\lambda(z)  = \lim_{\lambda \to 0^+}  \vartheta_{\lambda}(z) = \vartheta_{0}(z)
\end{align*}
where 
\begin{align*}
	\vartheta_{0} = \sup \left\lbrace w\in C(\Omega):H(x,Dw,0)\leq c_\Omega(H)\;\text{in}\;\Omega, \int_{\overline{\Omega}\times \R^n} \partial_uL(x,v,0)\cdot w(x)\;d\mu(x,v) \geq 0\;\forall\;\mu \in \mathfrak{M}(\Omega) \right\rbrace,
\end{align*}
where $\mathfrak{M}(\Omega)$ is the set of Mather measures in $\Omega$ (see Definition \ref{defn:MatherMeasuresOmega}).
This implies the conclusion that $u_\lambda\to u$ locally uniformly in $\R^n$.

\appendix
\label{appendix:appendix}
\section{Properties of the Lagrangian and Comparison principles} \label{appendix:Lagrangian}
We assume the following assumptions. 

\begin{description}[style=multiline, labelwidth=1cm, leftmargin=2cm]    
    \item[\namedlabel{itm:assumptions-k0}{$(\mathcal{K}_0)$}] Assume $H:\R^n\times \R^n$ be continuous. 
    \item[\namedlabel{itm:assumptions-k1}{$(\mathcal{K}_1)$}]
	For any compact subset $K\subset \R^n$ we have $\lim_{|\xi|\to \infty} \min_{x\in K} H(x,p) = + \infty$. 
    \item[\namedlabel{itm:assumptions-k2}{$(\mathcal{K}_2)$}] For each $x\in \R^n$, $p\mapsto H(x,p)$ is convex.
\end{description}

\subsection{Properties of the Lagrangian}
We define $L:\R^n\times \R^n \to \R$ by $L(x,v) = \sup_{p\in \R^n} \big(p\cdot v - H(x,p)\big)$ for $(x,v)\in \R^n\times\R^n$. 
It is standard that $v\mapsto L(x,v)$ is convex, and it can be infinite at some points. If $H$ is superlinear, i.e., $H(x,p)/|p| \to \infty$ as $|p|\to \infty$ then $L(x,v):\R^n\times\R^n\to \R$ is finite everywhere, and is also superlinear \cite[Theorem 2.16]{tran_hamilton-jacobi_2021}. \medskip

\begin{lem}[{\cite[Propositions 2.1 and 2.2]{ishii_asymptotic_2008}}]\label{lem:AppendixPropertiesL} Assume \ref{itm:assumptions-k0}--\ref{itm:assumptions-k2}. 
\begin{itemize}
	\item[(i)] $L$ is superlinear locally in $x$, i.e., for every compact subset $K\subset \R^n$, we have $\lim_{|v|\to\infty} \inf_{x\in K} \frac{L(x,v)}{|v|} = +\infty$. 

	\item[(ii)] For each $R>0$ there exists $C_R>0$ and $\delta_R>0$ such that $L(x,v) \leq C_R$ for $(x,v)\in B_R(0)\times B_{\delta_R}(0)$.
	\item[(iii)] Assume further that $p\mapsto H(x,p)$ is strictly convex, then $v\mapsto L(x,v)$ is continuously differentiable, i.e., $(x,v)\mapsto D_vL(x,v)$ is continuous. 
\end{itemize}
\end{lem}

The proof of the following Proposition is omitted due to its similarity to proofs found in \cite[Proposition 2.5]{ishii_asymptotic_2008}, \cite[Proposition 5.1]{mitake_asymptotic_2008}, or \cite[Lemma 2.10]{tu_generalized_2024}.

\begin{prop}\label{prop:curve} Assume \ref{itm:assumptions-k0}--\ref{itm:assumptions-k2}. Let $\Omega\subset \R^n$ be open, \emph{bounded}, and $\gamma\in \mathrm{AC}([a,b];\R^n)$ such that $\gamma([a,b])\subset \Omega$.
\begin{itemize}
\item[$\mathrm{(i)}$] If \( u \in C(\mathbb{R}^n) \) is locally Lipschitz and a viscosity subsolution to \( H(x, Du) \leq c(H) \) in \( \Omega \), then
    \begin{equation*}
        u(\gamma(b)) - u(\gamma(a)) \leq \int_a^b \Big( L(\gamma(s),\dot{\gamma}(s)) + c(H)\Big) \;ds. 
    \end{equation*}
\item[$\mathrm{(ii)}$] If \( u(x,t) \in C(\mathbb{R}^n \times [0,T]) \) is locally Lipschitz and a viscosity subsolution to \( u_t + H(x,Du) \leq 0 \) in \( \Omega \times (0,T) \), then for any \( 0 < a < b < T \), we have      
    \begin{equation*}
        u(\gamma(b), b) - u(\gamma(a), a) \leq \int_a^b  L(\gamma(s),\dot{\gamma}(s)) \;ds. 
    \end{equation*}
\end{itemize}
\end{prop}

\begin{lem}[{\cite[Lemma 6.3]{ishii_asymptotic_2008}}] \label{lem:minimizerUSC}
Assume \ref{itm:assumptions-k0}--\ref{itm:assumptions-k2}. Let $T>0$ and let $\{\gamma_k\}_{k\in \N}\subset \mathrm{AC}([0,T];\R^n)$ be a sequence such that 
\begin{align*}
	\sup_{k\in \N} \int_0^T L\big(\gamma_k(s), \dot{\gamma}_k(s)\big)\;ds < \infty. 
\end{align*}
Then, there exists a subsequence $\{\gamma_{k_j} \}_{j\in \N}$ and $\gamma\in \mathrm{AC}([0,T];\R^n)$ such that $\Vert \gamma_{k_j} - \gamma\Vert_{L^\infty([0,T])} \to 0$ as $j\to \infty$, and 
\begin{align}\label{eq:GammakLSC}
	\int_0^T L\big(\gamma(s), \dot{\gamma}(s)\big)\;ds \leq \liminf_{j\to \infty} \int_0^T L\big(\gamma_{k_j}(s), \dot{\gamma}_{k_j}(s)\big)\;ds.
\end{align}
\end{lem}

\subsection{Comparison principles}\label{appendix:ComaprisonPrinciples}

\begin{thm}\label{thm:Appendix:HJconstraint} Assume \ref{itm:assumptions-k0} and that \( \Omega \) satisfies \ref{itm:assumptions-O1}. 
\begin{description}[style=multiline, labelwidth=1cm, leftmargin=2cm]    
    \item[\namedlabel{itm:assumptions-O1}{$(\mathcal{O})$}] Let $\Omega\subset \R^n$ be open, bounded, connected such that
    here exists a universal pair $(r,h)\in (0,\infty)\times (0,\infty)$ and a uniformly bounded continuous function $\eta\in \mathrm{BUC}\left(\overline{\Omega};\mathbb{R}^n\right)$ such that
    \begin{equation}\label{eq:O1}
        B_{rt}\big(x+t\eta(x)\big) \subset \Omega \qquad\text{for all}\qquad x\in \partial \Omega, \quad t\in (0,h]. 
    \end{equation}
\end{description}
Let us consider 
\begin{equation}\label{eq:AppendixHJStatic}
    \lambda v(x) + H(x,Dv(x)) = 0 \qquad\text{in}\; \Omega. 
\end{equation}
Let $v_1\in \mathrm{BUC}(\overline{\Omega};\mathbb{R})$ be a viscosity subsolution of \eqref{eq:AppendixHJStatic} in $\Omega$ and $v_2\in \mathrm{BUC}(\overline{\Omega};\mathbb{R})$ be a viscosity supersolution of \eqref{eq:AppendixHJStatic} on $\overline{\Omega}$. If either \( v_1 \) or \( v_2 \) is Lipschitz, or \ref{itm:assumptions-k1} holds, then $ v_1(x)\leq v_2(x)$ for all $x\in \overline{\Omega}$.
\end{thm}

\begin{thm}[Modification of {\cite[Theorem 4.1]{ishii_asymptotic_2008}}] \label{thm:ComparisonIshiiMod} Assume \ref{itm:assumptions-k0}--\ref{itm:assumptions-k2}. Let $u,v \in C(\R^n\times [0,T))$ be continuous subsolution and supersolution to
\begin{equation}\label{eq:HJTime}
    u_t + H(x,Du) = 0 \qquad \text{in}\;\R^n\times (0,T).
\end{equation}
Assume there exists a constant $C>0$ and a Lipschitz function $\phi:\R^n\to \R$ such that 
\begin{equation*}
    \begin{cases}
    \begin{aligned}
        & H(x,D\phi(x)) \leq C \qquad\text{in the viscosity sense}\; x\in \R^n, \\
        & \displaystyle \lim_{r\to \infty} \inf 
        \left \lbrace 
            v(x,t) - \phi(x): |x|\geq r, t\in (0,T)
        \right \rbrace  = +\infty. 
    \end{aligned}
    \end{cases}
\end{equation*}
If $u(x,0)\leq v(x,0)$ for $x\in \R^n$, then $u(x,t)\leq v(x,t)$ for all $(x,t) \in \R^n\times [0,T)$. 
\end{thm}

\subsection*{Acknowledgments}
The work of Son Tu began at Michigan State University. Son Tu thanks Hung Tran and Hiroyoshi Mitake for valuable comments, and Olga Turanova for helpful discussions.


\end{document}